%% file: new-e-seq-def.tex
\documentclass[a4paper,12pt]{amsart} 
\newcommand{\amsart}[1]{{#1}}
\usepackage[vcentering,dvips]{geometry} %%% doesn't work within \amsart{..}
\newcommand{\svjour}[1]{}

\usepackage{xcolor}
\usepackage{amsmath}
\usepackage{amsfonts}
\usepackage{amssymb}
\usepackage{amsthm}
\usepackage{amsopn}
\usepackage{amscd}
\usepackage{pstricks}
\usepackage[utf8]{inputenc}
\usepackage{courier}
\usepackage{url}
\usepackage{color}
\usepackage{subfigure}
\usepackage{xypic}
\usepackage{graphicx}
\usepackage{paralist}
\usepackage[all,cmtip]{xy}

\amsart{
\bibliographystyle{alpha}
}
\svjour{
\bibliographystyle{apalike}
}

\entrymodifiers={+!!<0pt,\fontdimen22\textfont2>}

\title{Exceptional Sequences on Rational $\CC^*$-Surfaces}

\amsart{
\author[A. Hochenegger]{Andreas Hochenegger}
\address{Mathematisches Institut,
        Universit\"at zu K\"oln,
        Weyertal 86--90,
        50931 K\"oln, Germany}
\email{ahochene@math.uni-koeln.de}

\author[N.O.~Ilten]{Nathan Owen Ilten}
\address{Department of Mathematics, University of California, Berkeley CA 94720, USA}
\email{nilten@math.berkeley.edu}
}

\svjour{
\author{Andreas Hochenegger \and Nathan Owen Ilten}
\institute{A. Hochenegger \at Mathematisches Institut, Universit\"at zu K\"oln, Weyertal 86--90, 50931 K\"oln, Germany\\\email{ahochene@math.uni-koeln.de}
\and
N.O. Illten \at Department of Mathematics, University of California, Berkeley CA 94720, USA\\\email{nilten@math.berkeley.edu} }
}

\newcommand{\danksagung}{We would like to thank David Ploog and the anonymous referee for a number of helpful comments.}

\definecolor{lcolor}{rgb}{1.0,0.3,0.0}

\newcommand{\CC}{\mathbb{C}}
\newcommand{\QQ}{\mathbb{Q}}

\newcommand{\ZZ}{\mathbb Z}

\renewcommand{\AA}{\mathbb{A}}
\newcommand{\PP}{\mathbb{P}}

\newcommand{\base}{\mathcal B}

\newcommand{\M}{\mathcal M}
\newcommand{\D}{\mathcal D}
\newcommand{\E}{\mathcal E}

\newcommand{\tv}{\mathbb{TV}}
\newcommand{\tX}{\widetilde{X}}
\newcommand{\cX}{\mathcal{X}}
\newcommand{\mcP}{\mathcal P}
\newcommand{\CO}{\mathcal O}

\newcommand{\F}{\mathcal F}

\newcommand{\tot}{\mathbf{tot}}
\newcommand{\pim}{\pi^\circ}
\newcommand{\bpim}{\bar\pi^\circ}
\newcommand{\bpimpr}{\overline{(\pi')}^{\,\circ}}
\newcommand{\inv}[1]{\left(#1\right)^{-1}}

\newcommand{\Ex}{\mathcal{E}}
\newcommand{\Ts}{\mathcal{A}}
\newcommand{\defto}{\leadsto}
\newcommand{\tilt}{\mathcal{T}}

\newcommand{\coloneqq}{\mathrel{\mathop:}=}

\DeclareMathOperator{\Pic}{Pic}

\DeclareMathOperator{\Proj}{Proj}

\DeclareMathOperator{\rk}{rk}

\DeclareMathOperator{\spec}{Spec}
\DeclareMathOperator{\ord}{ord}

\DeclareMathOperator{\pr}{pr}

\DeclareMathOperator{\Aut}{Aut}
\DeclareMathOperator{\Hom}{Hom}
\DeclareMathOperator{\Ext}{Ext}

\DeclareMathOperator{\dCDiv}{-Div}
\DeclareMathOperator{\Aug}{Aug}
\newcommand{\TCDiv}{\mathbb{C}^*\hspace{-.15cm}\dCDiv}

\DeclareMathOperator{\Mod}{-mod}
\DeclareMathOperator{\EN}{End}
\DeclareMathOperator{\End}{End}
\DeclareMathOperator{\T}{\mathcal{T}}
\DeclareMathOperator{\dP}{dP}
\DeclareMathOperator{\coker}{coker}

\theoremstyle{theorem}
\newtheorem{theorem}{Theorem}[section]
\newtheorem{mainthm}{Main Theorem}
\newtheorem{cor}{Corollary}[theorem]
\newtheorem{prop}[theorem]{Proposition}
\newtheorem{lemma}[theorem]{Lemma}
\theoremstyle{definition}
\newtheorem{definition}[theorem]{Definition}
\newtheorem{rem}[theorem]{Remark}
\newtheorem{ex}[theorem]{Example}
\newtheorem*{ex*}{Example}
\newgray{gray1}{.7}
\newgray{gray2}{.74}
\newgray{gray3}{.78}
\newgray{gray4}{.82}
\newgray{gray5}{.86}
\newgray{gray6}{.9}
\newgray{gray7}{.94}
\newgray{gray8}{.98}

\newgray{newgray1}{.4}
\newgray{newgray2}{.45}
\newgray{newgray3}{.55}
\newgray{newgray4}{.6}
\newgray{newgray5}{.65}
\newgray{newgray6}{.7}
\newgray{newgray7}{.75}
\newgray{newgray8}{.8}
\newgray{newgray9}{.85}
\newgray{newgray10}{.9}
\newgray{newgray11}{.95}
\newgray{newgray12}{1}
\newcommand{\nocontract}{\bullet}
\newcommand{\contract}{\circ}

\include{newfigures}

\begin{document}

\begin{abstract}
Inspired by Bondal's conjecture, we study the behavior of  exceptional sequences of line bundles on rational $\CC^*$-surfaces under homogeneous degenerations. In particular, we provide a sufficient criterion for such a sequence to remain exceptional under a given degeneration. We apply our results to show that, for toric surfaces of Picard rank 3 or 4, all full exceptional sequences of line bundles may be constructed via augmentation. We also discuss how our techniques may be used to construct noncommutative deformations of derived categories.
\end{abstract}

\amsart{
}
\svjour{
}
\maketitle
\section*{Introduction}
In their study of exceptional sequences of line bundles on rational surfaces in \cite{hillperl08}, L.~Hille and M.~Perling introduced so-called \emph{toric systems}. A toric system on a rational surface  of Picard number $\rho$ consists of a $(\rho+2)$-tuple of divisor classes satisfying certain conditions regarding their intersection numbers, see Definition~\ref{def:toricsystem}.
 In any case, every full exceptional sequence of line bundles on a rational surface gives rise to a toric system. On the other hand, in \cite{hillperl08} it was shown  how for any toric system $\Ts$ on a rational surface $X$,  one can construct an associated smooth toric variety $\tv(\Ts)$, see Section \ref{sec:toricsystems}.  A.~Bondal loosely conjectured  that there should be a degeneration from $X$ to $\tv(\Ts)$.
As the following example shows, this cannot be true if interpreted in a naive sense.

\begin{ex*}
	We consider the second Hirzebruch surface $\F_2=\Proj_{\PP^1}(\CO\oplus\CO(2))$. A full exceptional sequence is given by by 
$\E=(\CO,\CO(P),\CO(Q),\CO(P+Q))$ 
where $P$ is the fiber class and $Q$ is the divisor class such that $P.Q=1$ and $Q^2=2$. The corresponding toric system is $\Ts=(P,-P+Q,P,-P+Q)$ which has associated the toric surface $\F_0=\PP^1\times\PP^1$. However, $\F_2$ does not degenerate (i.e. specialize) to $\F_0$. On the contrary, $\F_0$ degenerates to $\F_2$. Note that although $\E$ was exceptional, it was not strongly exceptional.
\end{ex*}

The example suggests two modifications to the conjecture:
either one restricts to toric systems coming from \emph{strongly} exceptional sequences, or one does not differentiate between degeneration (i.e. specialization) and deformation (i.e. generalization). 

In the present paper, we take a slightly different tack. Instead of trying to associate a degeneration to some toric system, we \emph{start} with both a degeneration and a toric system, and observe the behaviour of the toric system with respect to this degeneration. We do this within the context of rational surfaces with $\CC^*$-action, where combinatorial techniques can be utilized.

Let us describe our approach more precisely.
Using the techniques introduced by R. Vollmert and the second present author in \cite{ilten:09b}, one can explicitly construct homogeneous one-parameter families with rational $\CC^*$-surfaces as the fibers. Given such a family $\pi\colon\cX\to\base$, there is a canonically defined  isomorphism $\bpim\colon\Pic(\cX_s)\to\Pic(\cX_0)$ between the Picard groups of a general fiber $\cX_s$ and the special fiber $\cX_0$, see  \cite{ilten:11e}; this isomorphism arises by lifting invariant line bundles on $\cX_s$ to $\cX$ and then restricting to $\cX_0$. 
We say that $\cX_0$ homogeneously deforms to $\cX_s$, or conversely, that $\cX_s$ homogeneously degenerates to $\cX_0$.

Since a toric system is just a tuple of divisor classes  and all surfaces considered are smooth,
a toric system {$\Ts$} on $\cX_s$ yields a tuple of divisor classes on $\cX_0$ via $\bpim$, which we will denote by $\bpim(\Ts)$, abusing notation slightly.
Conversely, a toric  system on $\cX_0$ yields a tuple of divisor classes on $\cX_s$ via $(\bpim)^{-1}$.
Our first main result is then the following:

\begin{mainthm}\label{mainthm:1}
	Let $X$ and $X'$ be two smooth, complete rational $\mathbb{C}^*$-surfaces both with Picard number $\rho>2$ and let $\Ts$ be a toric system on $X$. Then there is a sequence $$X=X^0\dashrightarrow X^1\dashrightarrow\cdots\dashrightarrow X^k=X'$$ of homogeneous deformations and degenerations connecting $X$ and $X'$ such that if $\Ts_i$ is the image of $\Ts$ on $X^i$, $\Ts_i$ is a toric system. Furthermore $\tv(\Ts_i)=\tv(\Ts)$ for all $i$.\end{mainthm}

	We then proceed to apply this machinery to toric systems on Hirzebruch surfaces, where exceptional sequences of line bundle are well understood. In particular, we relate  degenerations of toric systems to so-called \emph{mutations}, see Definition \ref{def:mutation}. 

	Although every full exceptional sequence of line bundles defines a toric system, not every toric system comes from such a sequence; we call a toric system \emph{exceptional} if it can be constructed from an  exceptional sequence of line bundles of length $\rk K_0(X) = \rho+2$. An important observation is that in the above setting, an exceptional toric system may in fact  degenerate to a non-exceptional toric system, see Remark \ref{rem:hirz0}.	

	For a special subset of exceptional toric systems, we can say more. We call a toric system on a rational surface \emph{constructible} if it can be constructed from an exceptional toric system on a Hirzebruch surface using the inductive process of augmentation from \cite{hillperl08}, see Definition \ref{def:augmentation}. Such toric systems are automatically exceptional and full, that is, they come from a full exceptional sequence. We show that for any rational surface $X$
	of fixed Picard number and any toric surface $Y$ with equal Picard number, there exists a constructible toric system $\Ts$ on $X$ with $\tv(\Ts)=Y$. Given a constructible toric system on some $\CC^*$-surface along with a degeneration, we then formulate a condition of \emph{compatibility}, which can be checked recursively, see Definition \ref{def:compatible}. The following theorem makes clear the importance of this condition:

\begin{mainthm}\label{mainthm:2}
	Let $\pi$ be a homogeneous deformation of rational $\CC^*$-surfaces with general fiber $\cX_s$ and let $\Ts$ be a constructible toric system on $\cX_s$. Then $\Ts$ is compatible with $\pi$ if and only if $\bpim( \Ts)$ is a constructible toric system. In particular, if $\Ts$ is compatible with $\pi$, then $\bpim (\Ts)$ is a full exceptional toric system.
\end{mainthm}

In light of this result, a natural question is if in fact all exceptional toric systems are constructible. For Picard rank $\rho=2$, this is true by definition. Using the machinery described above, we can show the following, which gives us an explicit description of all full exceptional sequences of line bundles on rational $\CC^*$-surfaces of rank less than $5$:
\begin{mainthm}\label{mainthm:3}
Let $X$ be a toric surface of Picard rank $3$ or $4$. Then any toric system on $X$ is exceptional if and only if it is constructible.
In particular, any exceptional sequence of line bundles of length $\rk K_0(X)$ is full.
 \end{mainthm}

On the other hand, we construct a toric surface of Picard rank $5$ which has a nonconstructible exceptional toric system, see Example \ref{ex:nonconstr:5}. 
Finally, we also discuss how homogeneous geometric deformations may be used to construct noncommutative deformations, that is, parametrizations of derived categories of rational surfaces; several such parametrizations have already been described in \cite{perling09}.

We now describe the organization of this paper.  Section \ref{sec:prelim} contains basics on toric varieties and $\CC^*$-surfaces, as well as recalling the necessary results concerning homogeneous deformations adapted to our setting. In Section \ref{sec:def}, we further analyze homogeneous deformations of $\CC^*$-surfaces. Basic definitions for exceptional sequences, toric systems, and augmentation may be found in Section \ref{sec:exceptional}. In Section \ref{sec:exdef} we apply the machinery we have introduced and prove Main Theorems \ref{mainthm:1} and \ref{mainthm:2}. In Section \ref{sec:constr:ts}, we focus on constructible toric systems and prove our Main Theorem \ref{mainthm:3}.
Finally, Section \ref{sec:noncomdef} contains a brief discussion of noncommutative deformations. 
\amsart{
\\
\\
\noindent\emph{Acknowledgements}: \danksagung 
}

\section{Rational $\CC^*$-Surfaces}\label{sec:prelim}
In this section, we introduce notation and some necessary results concerning rational $\CC^*$-surfaces.  We begin by recalling some basic concepts from toric geometry.  We then describe rational $\CC^*$-surfaces in terms of multidivisors.  Finally, we related homogeneous deformations and degenerations of these surfaces to so-called degeneration diagrams.

\subsection{Toric Basics}
We begin by recalling some basics of toric geometry, see for example \cite{fulton:93a}.
Let $N$ be a lattice with dual $M$, and let $N_\QQ$ and $M_\QQ$ be the associated $\QQ$-vector spaces.
For any polyhedral subdivision $S$ in $N_\QQ$, and any $k\in \ZZ_{\geq 0}$, let $S(k)$ be the set of $k$-dimensional elements of $S$.
Consider a pointed polyhedral cone $\sigma\subset N_\QQ$. We say that $\sigma$ is \emph{smooth} if the primitive generators of rays of $\sigma$ can be completed to a lattice basis of $N$. Given a pointed polyhedral cone $\sigma\subset N_\QQ$, we associate an affine toric variety
$$
\tv(\sigma)\coloneqq\spec \CC[\sigma^\vee\cap M]
$$
whose dimension equals the rank of $N$. An affine toric variety is smooth exactly when the corresponding cone is smooth. 

Recall that a \emph{fan} $\Sigma$ is a set of cones forming a polyhedral subdivision. Given a fan $\Sigma$, we can construct a toric variety
$$
\tv(\Sigma)\coloneqq\bigcup_{\sigma\in \Sigma} \tv(\sigma)/\sim,
$$
where $\sim$ means that we glue the open subsets $\tv(\sigma)$ along common localizations in the fraction field of $\CC[M]$. We call a fan \emph{smooth} if all elements are smooth; we call a fan \emph{complete}  if the union of all elements is the vector space $N_\QQ$.  The variety $\tv(\Sigma)$ is smooth, respectively, complete, if and only if $\Sigma$ is smooth, respectively, complete.

Invariant prime divisors of $X=\tv(\Sigma)$ correspond to rays of $\Sigma$. For $\rho\in\Sigma{(1)}$, let $D_\rho$ denote the corresponding divisor. By an abuse of notation, we will use the same symbol to denote a ray and its primitive lattice generator. 
Suppose now that $X=\tv(\Sigma)$ is a smooth complete surface, i.e. $\Sigma$ is smooth and complete. Then we can order the rays of $\Sigma$ in some counterclockwise order $\rho_0,\ldots,\rho_l$. Let $-b_i$ be the self-intersection number of $D_{\rho_i}$. Then using cyclical notation we have  $b_i\rho_i=\rho_{i-1}+\rho_{i+1}$. In particular, up to lattice isomorphism, the numbers $b_i$ determine the fan $\Sigma$, and thus determine $X$ up to equivariant isomorphism.  We will thus also use the notation $\tv(b_0,\ldots,b_l)$ to denote $X$. Note that the Picard number $\rho(X)$ is equal to $l-1$, and we have the equality $\sum b_i=3l-9$.

Let $c_1,c_2,\ldots,c_k\in\mathbb{Z}$. The continued fraction $[c_1,c_2,\ldots,c_k]$ is inductively defined as follows if no division by $0$ occurs: $[c_k]=c_k$,  $[c_1,c_2,\ldots,c_k]=c_1-1/[c_2,\ldots,c_k]$. 
Now consider some $l+1$ tuple $(b_0,\ldots,b_{l})$ defining a smooth toric surface with corresponding fan $\Sigma$ in $N_\QQ$. Suppose that $b_0<0$ and $l>2$. Using induction on $l$, one can easily show that there exists a unique index $\alpha$, $1<\alpha<l$ such that $[b_1,\ldots,b_{\alpha-1}]$ is well defined and equals zero, or equivalently, that $\rho_\alpha=-\rho_0$. If we are in this situation, we define
\begin{equation}\label{eqn:gamma}
	\gamma=\sum_{i=1}^{\alpha-1}(3-b_i)-3.
\end{equation}
We will use the following lemma in the proof of Theorem \ref{thm:toricdef}:
\begin{lemma}\label{lemma:gamma}
	We always have $\gamma \geq 0$. Likewise, $b_0+b_\alpha-\gamma \geq 0$. Finally, for $R\in M$ such that $\langle \rho_\alpha ,R\rangle =1$, we have $\langle \rho_{\alpha-1},R\rangle-\langle \rho_{1},R\rangle=\gamma$.
\end{lemma}
	\begin{proof}
		All statements can be easily shown by induction on $l$.		\end{proof}
\subsection{Multidivisors and Rational $\CC^*$-Surfaces}\label{sec:multidiv}
By a \emph{$\CC^*$-surface}, we mean a complete surface with an effective action by the multiplicative group $\CC^*$. Thus, $\CC^*$-surfaces provide a  natural generalization of toric surfaces. 
These surfaces have been studied extensively, see for example \cite{orlik:77a} and \cite{MR2327238}.
It turns out that $\CC^*$-surfaces correspond to so-called \emph{multidivisors}.
\begin{definition}
	A \emph{multidivisor} $\M$ on a smooth projective curve $Y$ consists of proper polyhedral subdivisions $\M_P$ of $\QQ$ for every $P\in Y$ together with a pair  $(\M_-,\M_+)\in \{\contract,\nocontract\}^2$ such that:
	\begin{enumerate}
		\item Only finitely many $\M_P$ differ from the trivial subdivision induced by a single vertex at $0$;
		\item If $\M_-=\contract$, then $\sum_{P\in Y} \min \M_P(0)  < 0;$
		\item If $\M_+=\contract$, then $\sum_{P\in Y} \max \M_P(0)  > 0.$
	\end{enumerate}
	Here $\M_P(0)$ denotes the set of vertices of $\M_P$.
\end{definition}
The  $\CC^*$-surface $X(\M)$ corresponding to a multidivisor $\M$ may be constructed as in \cite[Proposition 1.6]{ilten:09d} by associating a so-called \emph{divisorial fan} to $\M$ which describes an open covering and gluing relations. The resulting surface is birational to the product $Y\times\PP^1$.
Describing this construction would take us too far afield, so we will instead give an ad hoc description of $X(\M)$ for the rational case, i.e. the case $Y=\PP^1$. We do this by gluing together localizations of toric varieties. As we will see below, the subdivisions $\M_P$ encode information about the fibers of a blowup $\tX(\M)$ of $X(\M)$, and the pair $(\M_-,\M_+)$ tells us what needs to be contracted to get back to $X(\M)$.

Consider $\M$ to be a multidivisor on $\PP^1$.
Let $\mcP\subset \PP^1$ be a finite set of order at least two containing all $P\in\PP^1$ with $\M_P$ nontrivial.
Let $U=\PP^1\setminus \mcP$, and for any  $P\in\mcP$, set $U_P=U
\cup 
\{P
\}$. Let $\phi_P$ be any automorphism of $\PP^1$ with $\phi_P(P)=0$ and $\infty\in\phi_P(\mcP)$.   Let $\Sigma_P$ be the fan generated by closures of cones of the form $\QQ_{\geq 0} \cdot (I,1)\subset\QQ^2$ for any interval $I\in\M_P(1)$, where $\M_P(1)$ denotes the set of one-dimensional polyhedron in $\M_P$. Projection of $\Sigma_P$ onto the second factor of $\QQ$ induces a map $\pr_P\colon\tv(\Sigma_P)\to\PP^1$.

Now, let $X_P=\pr_P^{-1}(\phi_P(U_P))$, and $Z_P=\pr_P^{-1}(\phi_P(U))$. Note that 
$$Z_P=\phi_P(U)\times\PP^1\cong U\times \PP^1.$$
We can thus construct a variety $\tX(\M)$ by gluing together the open charts $X_P$ along $Z_P$ via the above isomorphism. The $\CC^*$-action on the fibers of the maps $\pr_P$ are preserved by this gluing.
The cocharacter lattice of this torus $\CC^*$ can be naturally identified with $\ZZ$.
 There are two projective lines of fixed points under this action in $\tX(\M)$, one corresponding to the ray $\QQ_{\geq 0}$, the other to the ray $\QQ_{\leq 0}$. To construct $X(\M)$ from $\tX(\M)$, we contract the first, respectively, second of these curves if $\M_+=\contract$ or $\M_-=\contract$. This construction doesn't depend on the choice of $\mcP$ or the $\phi_P$. 

From the above description of $X(\M)$, it is easy to see what the invariant prime Weil divisors are, see also \cite{petersen:08a}. Indeed, there are divisors $D_+$ and $D_-$ consisting of fixpoints if $\M_+$, respectively $\M_-$ equals $\nocontract$, and for every $P\in \PP^1$ and $v\in\M_P$, there is a corresponding prime divisor $D_{P,v}$ lying in the fiber over $P$ of the rational map $X(\M)\dashrightarrow\PP^1$. From the above description, it is also clear that $X(\M)$ is toric if and only if $\mcP$ can be chosen to have order two, say $\mcP=\{P,Q\}$. In this case, $X(\M)=\tv(\Sigma)$, where $\Sigma$ is the complete fan with rays through $(v,1)$ for $v\in\M_P(0)$, $(v,-1)$ for $v\in \M_Q(0)$, and $(1,0)$ or $(-1,0)$ if $\M_+$ respectively $\M_-$ equals $\nocontract$. 

\begin{figure}[htbp]
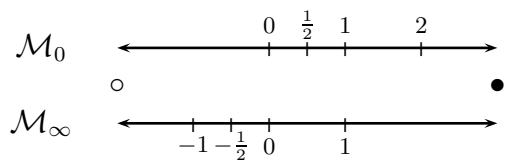

	\begin{center}	\exmulti
	\end{center}
    \caption{A possible multidivisor on $\PP^1$.}\label{fig:exmulti}
\end{figure}
\begin{ex}[A blowup of $\PP^2$]
	Consider the multidivisor pictured in Figure \ref{fig:exmulti}, where $\M_-$ and $\M_+$ respectively correspond to the dots on the left and right.
In this example, $X(\M)$ is the rational $\CC^*$-surface obtained by taking the toric variety $\PP^2$, blowing up in two of the three toric fixpoints, and then blowing up in the four resulting fixpoints of the exceptional divisors. Since this multidivisor only has two non-trivial slices, $X(\M)$ is in fact still toric.
\end{ex}

We now characterize multidivisors $\M$ on $\PP^1$ giving smooth $\CC^*$-surfaces.
 We say that $\M$ is \emph{smooth in the middle} if the cone  $\QQ_{\geq 0}\cdot(I,1)$ is  smooth  for all $P\in\PP^1$ and $I\in\M_P(1)$ compact. Suppose that $\M_+=\nocontract$. Then we say that $\M$ is \emph{smooth on the right} if $\QQ_{\geq 0}\cdot(I,1)$  is smooth if for all $P\in\PP^1$ and $I\in\M_P(1)$ unbounded on the right. Suppose instead that $\M_+=\contract$.  Then we say that $\M$ is \emph{smooth on the right} if there exists $f_+\in\CC(\PP^1)$ and $\phi_+\in\Aut(\PP^1)$ such that 
$$
\max \M_{\phi_+(P)}(0)+\ord_P(f_+)=0,\qquad P\neq 0,\infty
$$
and the cone
$$
\sigma_+=\langle (\max \M_{\phi_+(0)}(0),1),(\max \M_{\phi_+^{-1}(\infty)}(0),-1)\rangle\subset\QQ^2
$$
is smooth.\footnote{Such a pair $f_+,\phi_+$ induce an isomorphism of the neighborhood of the image of $D_+$ in $X(\M)$ with $\tv(\sigma_+)$.}
We define \emph{smooth on the left} similarly. Finally, we say that $\M$ is \emph{smooth} if it is smooth on the left, right, and in the middle.

\begin{prop}[cf. {\cite[Proposition 5.1 and Theorem 5.3]{liendo:10a}}]\label{prop:smooth}
Let $\M$ be a multidivisor on $\PP^1$.  Then the rational $\CC^*$-surface $X(\M)$ is smooth if and only if $\M$ is smooth.
\end{prop}

For the remainder of the article, we will only be dealing with smooth rational multidivisors.

\subsection{Degeneration Diagrams}
In \cite{ilten:09b}, R. Vollmert and the second author have shown how to construct homogeneous deformations of rational varieties with codimension-one torus action via certain combinatorial methods. For the case of present interest, i.e. smooth rational $
\CC^*$-surfaces, this combinatorial data may be encoded in a so-called \emph{degeneration diagram}:
\begin{definition}
Let $\M$ be a smooth rational multidivisor on $\PP^1$, and choose some $s\in\PP^1\setminus\{0,\infty\}$. A degeneration diagram for $\M$ consists of the pair $(\M,G)$, where $G$ is a connected graph on the vertices of $\M_0$ and $\M_s$ such that
\begin{enumerate}
\item $G$ is bipartite with respect to the natural partition induced by $\M_0$ and $\M_s$;
\item $G$ can be realized in the plane with all edges being line segments, after embedding $\M_0$ and $\M_s$ in parallel lines;
\item Every vertex of $G$ with valency strictly larger than one is a lattice point.
\end{enumerate}
\end{definition}

Given a degeneration diagram $(\M,G)$, we may construct a new smooth rational multidivisor $\M^{(0)}$ as follows. Let $\M_+^{(0)}=\M_+$, $\M_-^{(0)}=\M_-$, and $\M_P^{(0)}=\M_P$ for $P\neq 0,s$.  We take $\M_s^{(0)}$ to be the trivial subdivision of $\QQ$. Finally, we take $\M_0^{(0)}$ to be the complete subdivision of $\QQ$ with vertices of the form $v_0+v_s$, where $\overline{v_0v_s}$ is an edge of $G$. Note that the graph $G$ encodes an admissible \emph{Minkowski decomposition} of the subdivision $\M_0^{(0)}$, see for example \cite[Definition 1.3]{ilten:11e}.

Set $\base=\AA^1\setminus(\mcP\setminus \{0,s\})$, where $\mcP$ is the set of all $P\in\PP^1$ with $\M_P$ nontrivial. For $t\in \base$, $t\neq 0$, we may also construct a new smooth rational multidivisor $\M^{(t)}$ by exchanging $\M_s$ and $\M_t$.

\begin{prop}[{\cite[Theorem 4.4]{ilten:09b}}]
	Consider a degeneration diagram $(\M,G)$.  Then there is a flat family $\pi\colon\cX\to\base$ such that $\pi^{-1}(t)=X(\M^{(t)})$ for all $t\in\base$. In particular, $\pi^{-1}(s)=X(\M)$ and $\pi^{-1}(0)=X(\M^{(0)})$. 
\end{prop}

\begin{figure}[htbp]
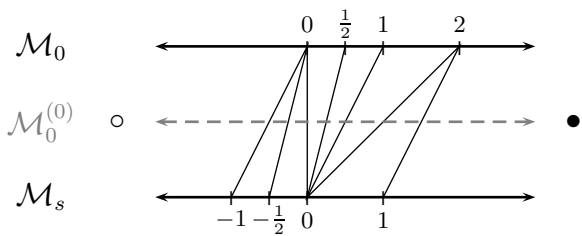

	\begin{center}	\exdegdi
	\end{center}
    \caption{A possible degeneration diagram}\label{fig:exdegdi}
\end{figure}
\begin{ex}[A blowup of $\PP^2$]\label{ex:degdi}
	Figure \ref{fig:exdegdi} presents a degeneration diagram for the multidivisor $\M$ of Figure \ref{fig:exmulti} with $s=\infty$. The resulting slice $\M_0^{(0)}$ can then be seen quite easily as the induced subdivision on the dashed line in between $\M_0$ and $\M_s$ scaled by a factor of two. The special fiber $X(\M^{(0)})$ of this degeneration is also a smooth toric surface.
\end{ex}

We will say that $X(\M)$ \emph{homogeneously degenerates} to $X(\M^{(0)})$, or equivalently, that $X(\M^{(0)})$ \emph{homogeneously deforms} to $X(\M)$. We will call $\pi^{-1}(s)=X(\M)$ and $\pi^{-1}(0)=X(\M^{(0)})$ respectively the general and special fibers of $\pi$, and denote them by $\cX_s$ and $\cX_0$. 
\begin{rem}
Note that if $\M$ has at most three non-trivial subdivisions $\M_P$, then all fibers $\pi^{-1}(t)$ are isomorphic for $t\neq 0$.
\end{rem}
The family $\pi$ induces a map of Picard groups $\bpim\colon\Pic(X(\M))\to\Pic(X(\M^{(0)}))$.  Indeed, any $\CC^*$-invariant divisor on $X(\M)$ can be canonically lifted to a $\CC^*$-invariant divisor on $\cX$ and then restricted to a $\CC^*$-invariant divisor on $X(\M^{(0)})$. This can be described explicitly.  For any $v\in\QQ$, let $\mu(v)$ be the smallest natural number such that $\mu(v)v\in\ZZ$. For any $\CC^*$-invariant divisor
\begin{equation*}
D=a_+D_++a_-D_-+\sum_{\substack{P\in\PP^1\\v\in\M_P(0)}} b_{P,v}D_{P,v}
\end{equation*}
on $X(\M)$, 
set
\begin{equation*}
\pim(D)=a_+D_++a_-D_-+\sum_{\substack{P\in\PP^1\setminus\{0,s\}\\v\in\M_P(0)}} b_{P,v}D_{P,v}+\sum_{\overline{v_0v_s}\in G}\mu(v_0+v_s)\left(\frac{b_{0,v_0}}{\mu(v_0)}+ \frac{ b_{s,v_s}}{\mu(v_s)}\right)  D_{0,v_0+v_s}.
\end{equation*}

\begin{prop}[{\cite{ilten:11e}}]
\label{prop:pim:preserves}
Any invariant divisor $D$  on $X(\M)$ can be canonically lifted to an invariant divisor $D^\tot$ on $\cX$ whose restriction to the special fiber is $\pim(D)$. The map $\pim$ is upper-semicontinuous with respect to cohomology, and preserves Euler characteristic, intersection numbers, canonical classes, and semiampleness. After factoring out by linear equivalence, $\pim$ induces an isomorphism of Picard groups $\bpim\colon\Pic(X(\M))\to\Pic(X(\M^{(0)}))$.
\end{prop}

\begin{rem}
	Let $D$ be an invariant divisor on $X(\M)$ as above.
	For $t\neq 0$, the restriction of $D^\tot$ to $\cX_t=\pi^{-1}(t)$ is just 
	\begin{equation*}
		D=a_+D_++a_-D_-+\sum_{v\in\M_s(0)} b_{s,v}D_{t,v}+\sum_{\substack{P\in\PP^1\setminus\{s\}\\v\in\M_P(0)}} b_{P,v}D_{P,v},
\end{equation*}
that is, divisors of the form $D_{s,v}$ are replaced with those of $D_{t,v}$.
\end{rem}
\begin{ex}
\label{ex:hirzedefdiv}
We look at an explicit example of the map $\bpim$, where $\pi\colon \cX \rightarrow \base$ is a deformation whose general fiber $\cX_s$ is a Hirzebruch surface $\F_r=\Proj_{\PP^1}(\CO\oplus\CO(r))$. As a toric variety, $\F_r$ corresponds to the complete fan $\Sigma_r$ with rays through $(1,0)$, $(0,1)$, $(-1,r)$, and $(0,-1)$. If $X(\M)=\F_r$ and $\M$ admits a non-trivial degeneration diagram, we can assume that $\M_0$ has vertices $-\frac{1}{r+\alpha}$, $0$ and that $\M_s$ has vertices $0,\frac{1}{\alpha}$ for some $\alpha>0$; this follows by a straightforward calculation from the description of $\Sigma_r$. We call this multidivisor $\M(r,\alpha)$. Note that with the exception of the case $r=0, \alpha= 1$, there is only one possible graph $G$ making $(\M(r,\alpha),G)$ into a degeneration diagram. Indeed, this is the bipartite graph where both $0$ vertices have degree two and the other two vertices have degree one. For the case $r=0, \alpha=\pm 1$, there is also the possibility of the bipartite graph $\tilde{G}$ where both $0$ vertices have degree one and the other two vertices, in this case lattice points, have degree two. In any case, the degeneration diagram $(\M(r,\alpha),G)$ (or $(\M(r,\alpha),\tilde{G})$) has corresponding special fiber $\F_{r+2\alpha}$. The difference between $G$ and $\tilde{G}$ corresponds to a flip on the total space of the deformation.

	For any Hirzebruch surface $\F_r$ with $r>0$, let $P$ be the divisor class of the fiber  of the ruling on $\F_r$, and let $Q$ be the unique class with $Q^2=r$ and $P.Q=1$. Now considering the isomorphism $X(\M(r,\alpha))\cong \F_r$, $P$ and $Q$ can respectively be represented by $D_{s,0}$ and $D_{s,{1}/{\alpha}}$. Consider now the deformation $\pi$ from $\F_{r+2\alpha}$ to $\F_r$ determined by the degeneration diagram $(\M(r,\alpha),G)$ and assume $r>0$. Then $\bpim(\CO(P))$ can be represented by $D_{0,{1}/\alpha}$ and $\bpim(\CO(Q))$ can be represented by $(r+\alpha)D_{0,-{1}/({r+\alpha})}+D_{0,0}$. One easily checks that
\begin{align}
	\bpim(\CO(P))&=\CO(P)\label{eqn:P}\\
	\bpim(\CO(Q))&=\CO(Q-\alpha P)\label{eqn:Q}
\end{align}
	where by abuse of notation, the $\CO(P)$ and $\CO(Q)$ on the right hand side of the equalities represent classes in $\Pic(\F_{r+2\alpha})$.

	The case of $r=0$ requires slightly more care, since there are two possible rulings on $\F_0$. Fix an isomorphism $\F_0\cong X(\M(0,\alpha))$ and consider the ruling of $\F_0$ given by the quotient map of the $\CC^*$-action on $X(\M(0,\alpha))$; note that this doesn't depend on $\alpha$. Then $P$ and $Q$ can be represented exactly as above. For $\pi$ corresponding to the degeneration diagram $(\M(0,\alpha),G)$, we once again have equations \eqref{eqn:P} and \eqref{eqn:Q}. On the other hand, for $\pi$ corresponding to the degeneration diagram $(\M(0,1),\tilde{G})$, we have $\bpim(\CO(P))=\CO(Q-P)$ and $\bpim(\CO(Q))=\CO(P)$. Thus, if in this case we instead consider the other possible ruling of $\F_0$ (and thus swap $P$ and $Q$), we once again have equations \eqref{eqn:P} and \eqref{eqn:Q}.
\end{ex}

\section{Properties of Homogeneous Deformations}\label{sec:def}
In this section, we will analyze the flat families coming from degeneration diagrams.  First, we will consider certain deformations where both the special and general fibers are toric. Then, we will consider the behaviour of homogeneous deformations with respect to blowing up and blowing down invariant curves. Then, we will show that these deformations suffice to connect all smooth rational $\CC^*$-surfaces of fixed Picard number larger than two.
\subsection{Deformations with Toric Fibers}
Consider some homogeneous deformation $\pi$ coming from a degeneration diagram. In general, even if $\cX_0$ is toric, $\cX_s$ may not be.  However, there are some homogeneous deformations for which both the special and general fiber are toric.

\begin{theorem}\label{thm:toricdef}
Consider a smooth complete toric surface $\cX_0=\tv(b_0,\ldots,b_l)$ such that $b_0<0$ and $l>2$. Then there exists a homogeneous deformation     of $\cX_0$ with toric general fiber
                                $$\cX_s=\tv(b_0+\gamma+2r,b_{\alpha-1},\ldots,b_{1},b_\alpha -\gamma-2r, b_{\alpha+1},\ldots,b_l),$$
                        where $\alpha$ and $\gamma$ are as in Lemma \ref{lemma:gamma} and $0\leq r\leq-b_0$.
\end{theorem}
\begin{figure}[htbp]
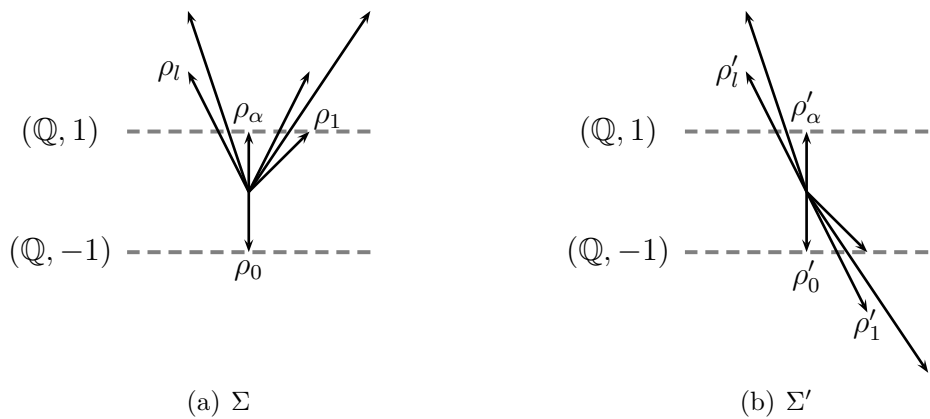

    \centering
    \subfigure[$\Sigma$]{\prooffan}
    \subfigure[$\Sigma'$]{\prooffanb}
    \caption{Possible fans from the proof of Theorem \ref{thm:toricdef}}\label{fig:prooffans}
\end{figure}
\begin{proof}
        Let $\cX_0=\tv(\Sigma)$ for some fan $\Sigma\subset \ZZ^2\otimes \QQ$ with rays $\rho_i$ corresponding to the numbers $b_i$, see for example Figure \ref{fig:prooffans}(a). After a lattice transformation, we may assume that $\rho_0=(0,-1)$ and $\rho_1=(1,r)$.
We construct a multidivisor $\M^{(0)}$ by taking $\M_0^{(0)}$ to be the subdivision induced by $\Sigma$ on the line $(\QQ,1)$, and by taking $\M_+^{(0)},\M_-^{(0)}$ to be $\nocontract$ if and only if $r=0$ or $r=-b_0$, respectively.\footnote{The condition $r=-b_0$ is equivalent to $\rho_l=(-1,0)$.} Note that the requirement $0\leq r \leq -b_0$ is exactly such that $\Sigma$ subdivides the line $(\QQ,-1)$ only with the ray $\rho_0$. Thus, we have $\cX_0=X(\M^{(0)})$ from the discussion in Section \ref{sec:multidiv}.

        We now construct a degeneration diagram $(\M,G)$ giving special fiber $\cX_0$. Indeed, let the vertices of $\M_0$ consist of those vertices  $v\in\M_0^{(0)}$ with $v\leq 0$. Likewise, let the vertices of $\M_s$ consist of those vertices $v\in\M_0^{(0)}$ with $v\geq 0$. Furthermore, take $\M_+=\M_+^{(0)}$, $\M_-=\M_-^{(0)}$. We then define $G$ to be the graph having edges $\overline{v_0v_s}$ with $v_0\in\M_0,v_s\in\M_s$ and either $v_0=0$ or $v_s=0$. One easily confirms that $(\M,G)$ is indeed a degeneration diagram with $\M_0^{(0)}$ being the multidivisor for the corresponding special fiber. 

	Now, we see that the general fiber $\cX_s=X(\M)$ is toric, since $\M_0$ and $\M_s$ are the only two nontrivial subdivisions. In fact, by embedding $\M_0$ and $\M_s$ in height one and minus one respectively, and possibly adding horizontal rays coming from $\M_+$ and $\M_-$, we recover a fan $\Sigma'$ with $\cX_s=X(\M)=\tv(\Sigma')$, see for example Figure \ref{fig:prooffans}(b). $\Sigma'$ then has rays $\rho_0',\ldots,\rho_l'$ ordered cyclically with $\rho_i'=\rho_i$ for $\alpha\leq i \leq l$ or $i=0$ and $\rho_i'$ a vertical reflection of $\rho_{\alpha-i}$ for $0<i<\alpha$.\footnote{Note that if $r=0$, then $\rho'_{\alpha-1}=\rho_1=(1,0)$, these rays corresponding to $\M_+=\M_+^{(0)}=\nocontract$.} Let $-b_i'$ be the self-intersection number of the divisor corresponding to the $\rho_i'$; then $\cX_s$ is represented by the chain $(b_0',\ldots,b_l')$. Now, it is immediate from this description that $b_i'=b_{\alpha-i}$ for $1\leq i < \alpha$ and that $b_i'=b_i$ for $\alpha<i\leq l$. Furthermore, if we take $R=[0,1]\in\Hom(\ZZ^2,\ZZ)$, we have 
                \begin{equation}\label{eqn:balpha}
                b_\alpha'=\langle \rho_{\alpha-1}',R\rangle+\langle \rho_{\alpha+1}',R\rangle=-\langle \rho_{1},R\rangle+\langle \rho_{\alpha+1},R\rangle.
        \end{equation}          
We also have 
$\langle \rho_{\alpha-1},R\rangle=r+\gamma$ by Lemma \ref{lemma:gamma}.
Since $b_\alpha=\langle \rho_{\alpha-1},R\rangle+\langle \rho_{\alpha+1},R\rangle$, we can then rewrite equation \eqref{eqn:balpha} as
\begin{equation*}
b_\alpha'=
-r+b_\alpha-\langle \rho_{\alpha-1},R\rangle=b_\alpha-\gamma-2r.
\end{equation*}
The sum of all the intersection numbers must remain constant, so we also have $b_0'=b_0+\gamma+2r$, completing the proof.
\end{proof}

\subsection{Blowing Up and Down Deformations}
We shall now see how homogeneous deformations are compatible with blowing up and blowing down.  We first need the following lemma:

\begin{lemma}
\label{lem:exdiv:valone}
        Let $(\M,G)$ be a degeneration diagram. For any edge $\overline{v_0v_s}$ of $G$ such that the divisor $D_{0,v_0+v_s}$ on $X(\M^{(0)})$ has self-intersection $-1$, one of the vertices $v_0,v_s$ must have valency one. 
\end{lemma}
\begin{proof}
        If $v_0+v_s$ lies to the left or to the right of all other vertices of $\M_0^{(0)}$, then the edge $\overline{v_0v_s}$ lies to the left or right of all other edges of $G$ as well; it is then clear that either $v_0$ or $v_s$ must have valency one. If on the other hand $v_0+v_s$ has left and right neighboring vertices $v',v''\in\M_0^{(0)}$, then $\mu(v_0+v_s)=\mu(v')+\mu(v'')>1$, where the equality follows from the toric formula for self-intersection numbers.
If neither $v_0$ nor $v_s$ has valency one, they must both be lattice points, in which case $\mu(v_0+v_s)=1$, a contradiction.
\end{proof}

Using this lemma, it is clear how to \emph{blow down} any homogeneous deformation. Indeed, let $\pi$ correspond to the degeneration diagram $(\M,G)$, and let $\phi\colon\cX_0\to \cX_0'$ be the contraction of an invariant minus one curve.

Suppose first of all that this curve is of the form $D_{P,v}$ for $P\neq 0$
or of the form $D_+$ or $\D_-$. Then we get a new multidivisor $\M'$ by respectively removing the vertex $v$ from the subdivision $\M_P$ or by setting $\M_+$ or $M_-$ to $\contract$. Setting $G'=G$, we then have that $(\M',G)$ is a degeneration diagram with $X({\M'}^{(0)})=\cX_0'$ and with $\cX_s'=X(\M')$ a blowdown of $\cX_s$.

On the other hand, suppose that $\phi$ blows down a curve of the form $D_{0,v}$ in $\cX_0$. Then $v$ corresponds to an edge $\overline{v_0v_s}$ of $G$ and by the above lemma, either $v_0$ or $v_s$ must have valency one; assume without loss of generality that this is $v_0$. We then get a new multidivisor $\M'$ by removing the vertex $v_0$ from the subdivision $\M_0$. Furthermore, we have a graph $G'$ on $\M'$ attained from $G$ by removing the edge $\overline{v_0v_s}$. Due to the fact that $v_0$ had valency one in $G$, one easily checks that $(\M',G')$ is a degeneration diagram.
As in the other case, we have  $X({\M'}^{(0)})=\cX_0'$ and $\cX_s'=X(\M')$ a blowdown of $\cX_s$. In this manner we define the \emph{blowdown of $(\M,G)$ by $\phi$} to be $(\M',G')$. We call the flat family corresponding to $(\M',G')$ the blowdown of $\pi$ by $\phi$.

It is also possible to lift a homogeneous deformation $\pi\colon\cX\to \base$ by an invariant \emph{blowup} $\phi$ of either the special fiber $\cX_0$ or the general fiber $\cX_s$. Indeed, let $(\M,G)$ be the corresponding degeneration diagram.

The first possible type of blowup of $\cX_0$ or $\cX_s$ is by blowing up in a fixpoint of the $\CC^*$-action which is contained in the closure of multiple orbits, that is, by 
changing $\contract$ to $\nocontract$ for either $\M_+$ or $\M_-$. If we define $\M'$ to be equal to $\M$ with the relevant modification of $\M_+$ or $\M_-$, we get a degeneration diagram $(\M',G')$ with either $X(\M'^{(0)})$ or $X(\M')$ the desired blowup of $\cX_0$ or respectively $\cX_s$.
Suppose instead that the blowup of $X_0$ or $X_s$ corresponds to inserting a vertex $v$ in the subdivision $\M_P^{(0)}=\M_P$ for $P\neq 0,s$. Then if we define $\M'$ to come from $\M$ by adding the vertex $v$ to $\M_P$ and setting $G=G'$, we get a degeneration diagram $(\M',G')$ with the same property as in the previous case.

Suppose now that a blowup of $\cX_0$ corresponds to inserting a vertex $v$ in the subdivision $\M_0^{(0)}$. This corresponds to the insertion of a vertex $\tilde{v}$ in either $\M_0$ or $\M_s$, which in turn corresponds to a blowup of $\cX_s$.\footnote{Note that the placement of the vertex in either $\M_0$ or $\M_s$ is uniquely determined if $v$ isn't an extremal vertex of $\M_0^{(0)}$.}  So assume that we have a blowup of $\cX_s$ of this form. Then we can define a multidivisor $\M'$ from $\M$ similar to the previous cases. Likewise, we can define a graph $G'$ on the vertices of $\M_0',\M_s'$ by adding an edge between $\tilde{v}$ and the unique vertex connected to all neighboring vertices of $\tilde{v}$.
This defines a degeneration diagram $(\M',G')$ with the same property as above. In all such cases, we call $(\M',G')$ a \emph{blowup of $(\M,G)$ by $\phi$}.

We can sum up the above discussion by the following proposition:
\begin{prop}\label{prop:blowupblowdown}
Let $(\M,G)$ be a degeneration diagram with corresponding special fiber $\cX_0$ and general fiber $\cX_s$.
        \begin{enumerate}
                \item \label{prop:blowupblowdown1} If $\phi\colon\cX_0\to \cX_0'$ is a blowdown of an invariant curve, there is a unique degeneration diagram $(\M',G')$ called the blowdown of $(\M,G)$ by $\phi$ such that $X(\M'^{(0)})=\cX_0'$ and $X(\M')$ is an invariant blowdown of $\cX_s$.
                \item if $\phi\colon\cX_0'\to \cX_0$ is an invariant blowup, there is a degeneration diagram $(\M',G')$ called a blowup of $(\M,G)$ by $\phi$ such that  $X(\M'^{(0)})=\cX_0'$ and $X(\M')$ is an invariant blowup of $\cX_s$.
                \item \label{prop:blowupblowdown3} if $\phi\colon\cX_s'\to \cX_s$ is an invariant blowup, there is a unique degeneration diagram $(\M',G')$ called the blowup of $(\M,G)$ by $\phi$ such that  $X(\M')=\cX_s'$ and $X(\M'^{(0)})$ is an invariant blowup of $\cX_0$.
        \end{enumerate}
\end{prop}

\begin{ex}
	In Figure \ref{fig:blowupblowdown}, we picture some possible blowups and blowdowns of the degeneration diagram from Figure \ref{fig:exdegdi}. In (a), we show the resulting diagram for the blowdown of $\cX_0$ in the divisor corresponding to $1/2\in\M_0^{(0)}$. In (b) and (c) we show the two possible ways of blowing up the degeneration coming from a blowup of $\cX_0$ at the point $2\in\M_0^{(0)}$.
\end{ex}

\begin{figure}[htbp]
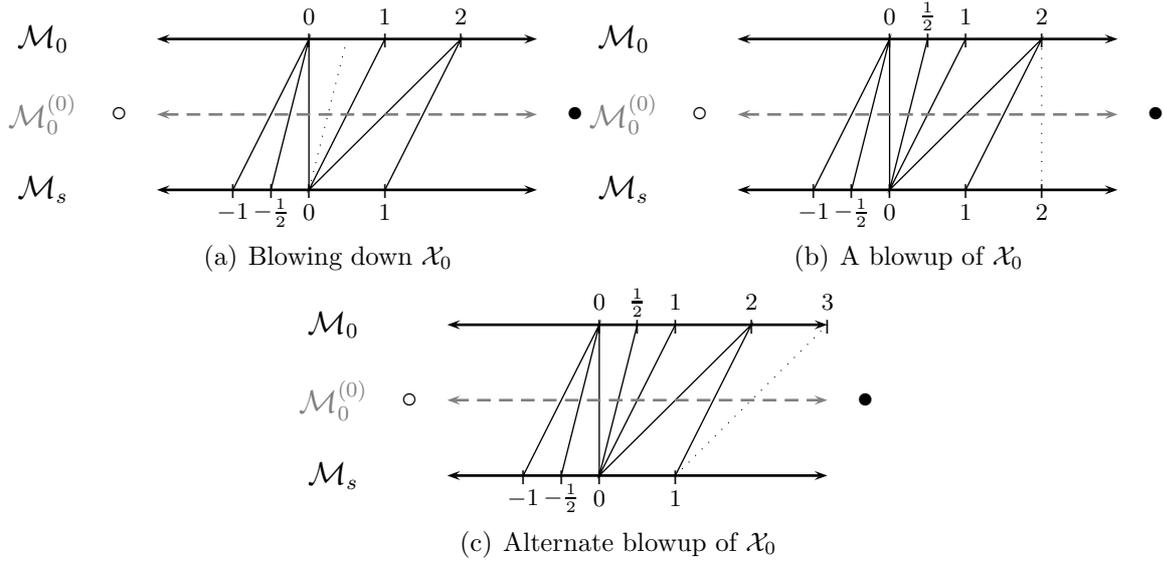

\begin{center}
	\subfigure[Blowing down $\cX_0$]{\exdegdibd}
	\subfigure[A blowup of $\cX_0$] {\exdegdibua}
	\subfigure[Alternate blowup of $\cX_0$]{\exdegdibub}
\end{center}
	\caption{Blowup and blowdown of degenerations diagrams.}\label{fig:blowupblowdown}
\end{figure}

The operations of blowing up and blowing down degeneration diagrams behaves nicely with the induced maps on invariant divisors. More specifically, in the above setting, let $\phi_0\colon\cX_0\to \cX_0'$ be an invariant blowdown of a minus one curve with $E_0$ the corresponding exceptional divisor. Let $\pi'$ be the blowdown of $\pi$ by $\phi$, with $\cX_s'$ the general fiber of $\pi'$. From the description of the blowdown of a degeneration diagram, one easily confirms that we have an invariant blowdown $\phi_s\colon\cX_s\to \cX_s'$; let $E_s$ be the corresponding exceptional divisor.

\begin{prop}
\label{prop:blowuppicommutes}
In the above situation, $\pim(E_s)=E_0$ and the following diagram commutes:
\[
\xymatrix{
\TCDiv(\cX_0) & \TCDiv(\cX_s) \ar[l]_{\pim{}}\\
\TCDiv(\cX_0') \ar[u]^{\phi_0^{*}}  & \TCDiv(\cX_s') \ar[u]_{\phi_s^*} \ar[l]^{(\pi')^\circ}
}
\]
Here, $\TCDiv$ denotes the group of $\CC^*$-invariant divisors.
\end{prop}
\begin{proof}
The claims follow from the description of the blowdown of a degeneration diagram, the description of $\pim$, and a straightforward calculation.
\end{proof}

\subsection{Deformation Connectedness}
Let $X$ and $X'$ be two smooth rational $\CC^*$-surfaces.
\begin{definition}\label{def:tdefcon}
        We say that $X$ and $X'$ are \emph{homogeneously deformation connected} if there is a finite sequence $X=X^0,X^1,\ldots,X^k=X'$ with $X^i$ homogeneously degenerating or deforming to $X^{i-1}$ for each $1\leq i \leq k$.
\end{definition}

It is well-known that a Hirzebruch surface of even parity cannot be deformed or degenerated to a Hirzebruch surface of odd parity and vice versa. An obstruction to such a deformation can be found by comparing the Chow rings. If we instead consider rational surfaces of fixed Picard number $\rho>2$, it is an easy exercise to see that all the Chow rings are isomorphic. Thus, the obstruction to deformation we had for the case $\rho=2$ no longer exists. In fact, for rational $\CC^*$-surfaces it is sufficient to consider homogeneous deformations:

\begin{theorem}\label{thm:defcon}
Consider the set of all smooth rational $\CC^*$-surfaces with Picard number $\rho$ for any integer $\rho > 2$. All elements of this set are homogeneously deformation connected. 
\end{theorem}

The proof of this theorem will constitute the remainder of this section. We first prove the following lemma:

\begin{lemma}\label{lemma:degentotoric}
Any smooth rational $\CC^*$-surface $X$ can be degenerated to a smooth toric surface via a finite number of homogeneous degenerations.
\end{lemma}
\begin{proof}
        Let $\M$ be a multidivisor with $X=X(\M)$. Suppose that $\M$ has more than three non-trivial subdivisions. Then there are non-trivial subdivisions $\M_P,\M_Q$ with $P\neq Q$ such that the left-most vertex $w_P$ of $\M_P$ and the right-most vertex $w_Q$ of $\M_Q$ are lattice points; this follows from the smoothness criterion of Proposition \ref{prop:smooth}. Setting $0=P,s=Q$ and considering the graph $G$ on the vertices of $\M_P,\M_Q$ with edges of the form $\overline{v_0w_Q}$ and $\overline{v_sw_P}$ for $v_0\in\M_P$, $v_s\in\M_Q$ gives a degeneration diagram $(\M,G)$. The multidivisor for the special fiber has one less non-trivial subdivision than $\M$.

        We can apply the above procedure inductively, and can thus assume that $\M$ has at most three non-trivial subdivisions. If $\M$ has less than three non-trivial subdivisions, then $X(\M)$ is toric, and we are done. If as above there are non-trivial subdivisions $\M_P,\M_Q$ with $P\neq Q$ such that the left-most vertex $v_P$ of $\M_P$ and the right-most vertex $v_Q$ of $\M_Q$ are lattice points, then we can once again proceed as above and degenerate to something with only two non-trivial subdivisions. We thus must only consider the remaining case, which is that where $\M$ has three non-trivial subdivisions $\M_0,\M_1,\M_\infty$ and $\M_0,\M_\infty$ have no extremal lattice vertices and both extremal vertices of $\M_1$ are lattice points. We show that this is actually impossible.

        In this case, we can actually assume that the left-most vertex of $\M_1$ is $0$, and that the right-most vertex is $n$. Let $u_0^l/v_0^l$, $u_\infty^l/v_\infty^l$ be the left-most vertices of $\M_0$ and $\M_\infty$ written in lowest terms and let $u_0^r/v_0^r$, $u_\infty^r/v_\infty^r$ similarly be the right-most vertices. Due to smoothness we have
\begin{align}
        -u_0^lv_\infty^l-u_\infty^lv_0^l&=1;\label{eqn:smooth1}\\
        u_0^rv_\infty^r+u_\infty^rv_0^r+v_0^rv_\infty^r\cdot n &= 1.\label{eqn:smooth2}
\end{align}
Furthermore, we of course have
\begin{equation}\label{eqn:smooth3}
        u_P^lv_P^r\leq u_P^rv_P^l\end{equation}
        for $P=0,\infty$. Solving equations \eqref{eqn:smooth1} and  \eqref{eqn:smooth2} for $v_0^l$ and $v_0^r$, substituting for these expressions in \eqref{eqn:smooth3} for $P=0$, and rearranging terms gives us
        \begin{equation*}
v_0^rv_\infty^r+v_0^lv_\infty^l+u_\infty^lv_\infty^rv_0^rv_0^l\geq u_\infty^rv_\infty^lv_0^lv_0^r+v_0^lv_\infty^lv_0^rv_\infty^r n.
\end{equation*}
Combining this with \eqref{eqn:smooth3} for $P=\infty$ then gives us
\begin{equation*}
v_0^rv_\infty^r+v_0^lv_\infty^l\geq v_0^lv_\infty^lv_0^rv_\infty^r n.
\end{equation*}
This however is a contradiction, since $n\geq 1$ and $v_0^l,v_\infty^l,v_0^r,v_\infty^r\geq 2$. Thus, this case never arises and we can always degenerate to a toric surface.
\end{proof}

In general, one can always construct a smooth rational surface of Picard rank higher than two by iteratively blowing up a Hirzebruch surface in a number of points. This can in fact be done equivariantly for smooth rational $\CC^*$-surfaces. For multidivisors $\M$ with $\M_+=\M_-=\nocontract$, this is stated in \cite{orlik:77a}. However, we know of no proof of the general case and thus provide one here as an easy corollary of the above lemma:

\begin{cor}\label{cor:equiblowup}
Any smooth rational $\CC^*$-surface $X$  with Picard number larger than two can be constructed from a Hirzebruch surface by a series of equivariant blowups.
\end{cor}
\begin{proof}
        Suppose that $X=\cX_s$ isn't a Hirzebruch surface. By Lemma \ref{lemma:degentotoric}, we know that $X$  degenerates to some toric variety $\cX_0$ via a chain of homogeneous deformations. But there is an invariant minus one curve on $\cX_0$ which can be blown down, since $\cX_0$ is toric, see \cite{fulton:93a}. Blowing down the deformations from $\cX_0$ to $\cX_s$ as in Proposition \ref{prop:blowupblowdown}  gives us a new general fiber $\cX_s'$ which is an invariant blowdown of $\cX_s$. The proof then follows by induction on the Picard number.
\end{proof}

We will collect several more lemmata which we shall need:

\begin{lemma}\label{lemma:nbdefcon}
Consider a smooth fan $\Sigma$ with rays $\rho_0,\ldots,\rho_l$. Let $\Sigma_1$ and $\Sigma_2$ be the smooth fans attained by inserting a ray between $\rho_0$ and $\rho_1$ respectively $\rho_1$ and $\rho_2$. Then $\tv(\Sigma_1)$ is homogeneously deformation connected to $\tv(\Sigma_2)$.
\end{lemma}

\begin{figure}[htbp]
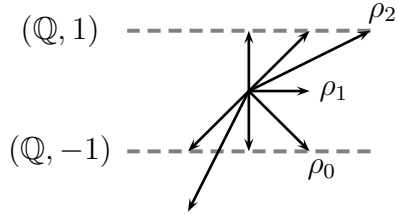

        \begin{center}
        {\buprooffan}
        \end{center}
        \caption{Example fan in proof of Lemma \ref{lemma:nbdefcon}}\label{fig:lemmaproof}
\end{figure}

\begin{proof}
        As in the proof of Theorem \ref{thm:toricdef}, we transform $\Sigma$ into a multidivisor $\M$. After a lattice transformation, we may assume that $\rho_1=(1,0)$, $\rho_0$ lies in height $-1$, and $\rho_1$ lies in height $1$.
For $P_1\neq P_2\in \PP^1$, let $\M_{P_1}$ and $\M_{P_2}$  be the subdivisions induced by $\Sigma$ on the affine lines $(\QQ,1)$ and $(\QQ,-1)$. $\M_+=\nocontract$, and $\M_-=\contract$ unless $(-1,0)$ is in $\Sigma(1)$. See for example Figure \ref{fig:lemmaproof}.

        Now, for some $s\in \PP^1\setminus\{P_1,P_2\}$, let $\overline{\M}$ be the multidivisor with $\overline{\M}_{P}=\M_P$ for $P\neq s$, and $\overline{\M}_s$ the subdivision of $\QQ$ with vertices $0$ and $1$. For $i=1,2$, let $G_i$ be the graph on the vertices of $\overline{\M}_s$ and $\overline{\M}_{P_i}$ with edges $\overline{v_s w}$ for either vertices $v_s=0\in\overline{\M}_s$ and $w\in \overline{\M}_{P_i}$ or vertices $v_s=1\in\overline{\M}_s$ and $w$ the right-most vertex in $\overline{\M}_{P_i}$. Setting $P_i=0$, one easily checks that $(\overline{\M},G_i)$ is a degeneration diagram with general fiber $X(\overline{\M})$ and special fiber $\tv(\Sigma_i)$, see for example Figure \ref{fig:lemmaproofb}. Thus, we have homogeneous deformations from both $\tv(\Sigma_1)$ and $\tv(\Sigma_2)$ to some common rational $\CC^*$-surface, making them homogeneously deformation connected.
\end{proof}

\begin{figure}[htbp]
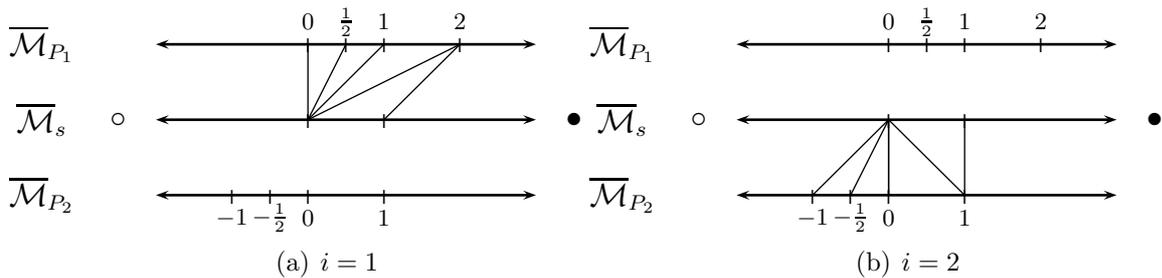

        \begin{center}
        \subfigure[$i=1$]{\proofdga}
        \subfigure[$i=2$]{\proofdgb}
        \end{center}
        \caption{Example degeneration diagrams in proof of Lemma \ref{lemma:nbdefcon}}\label{fig:lemmaproofb}
\end{figure}

\begin{rem}
        The two deformations constructed in the above proof can be naturally glued together to give a flat family $\cX$ over $\PP^1$ with fibers $\cX_0=\tv(\Sigma_1)$ and $\cX_\infty=\tv(\Sigma_2)$. In this family, the fiber over any point $s\in\PP^1$ is simply the blowup of $\tv(\Sigma)$ in $s$, where we have identified the base space $\PP^1$ with the divisor corresponding to the ray $\rho_1$.
\end{rem}

\begin{lemma}\label{lemma:dckn}
	The set $\left\{\tv(b_0,0,b_\alpha,1,1)\ |\ b_0+b_\alpha=1\right\}$ consists of all (smooth, projective) toric surfaces of Picard rank $3$ and is homogeneously deformation connected.
\end{lemma}
\begin{proof}
        The Hirzebruch surface $\F_r$ corresponds to the tuple $(r,0,-r,0)$. Since any rank $3$ toric surface may be attained by an invariant blowup from a Hirzebruch surface, these correspond exactly to the tuples $(r+1,1,1,-r,0)$ and $(r,1,1,-r+1,0)$ up to cyclic permutation. Thus, rank $3$ toric surfaces are exactly of the desired form.

	Consider now $b_0,b_\alpha$ and $b_0',b_\alpha'$ such that $b_0+b_\alpha=b_0'+b_\alpha'=1$. Due to symmetry we can assume that $b_0,b_0'$ have the same parity. Let $b_0''=-\max\{ |b_0|,|b_0'|\}$. Then $\tv(b_0'',0,n-b_0'',1,1)$ deforms to both $\tv(b_0,0,b_\alpha,1,1)$ and to $\tv(b_0',0,b_\alpha',1,1)$ by Theorem \ref{thm:toricdef}, so the desired set is homogeneously deformation connected.
\end{proof}

We now turn to the proof of the above theorem:

\begin{proof}[Proof of Theorem \ref{thm:defcon}]
        We will prove the theorem by induction on $\rho$.  From Lemma \ref{lemma:degentotoric} we have that any smooth rational $\CC^*$-surface can be homogeneously degenerated to a smooth toric surface. Thus, for $\rho=3$ the statement then follows from Lemma \ref{lemma:dckn}.

        Assume that the theorem holds for Picard number $\rho$, and consider any two smooth rational $\CC^*$-surfaces $X^1,X^2$ with Picard number $\rho+1$. By again applying Lemma \ref{lemma:degentotoric}, we can assume without loss of generality that $X^1$ and $X^2$ are toric. Let $\tilde{X}^i$ be an invariant blowdown of $X^i$. Then $\tilde{X}^1$ and $\tilde{X}^2$ are homogeneously deformation connected by the induction hypothesis, and this series of deformations and degenerations can be blown up to connect $\hat{X}^1$ and $\hat{X}^2$, where $\hat{X}^i$ is an invariant blowup of $\tilde{X}^i$. Thus, we must only show that $\hat{X}^i$ and $X^i$ are homogeneously deformation connected, that is, any two invariant blowups in a point of a common toric surface are homogeneously deformation connected. But this follows from repeated application of Lemma \ref{lemma:nbdefcon}, proving the theorem.
\end{proof}

\section{Exceptional Sequences and Toric Systems}\label{sec:exceptional}
We first recall basic definitions for exceptional sequences and toric systems in Section \ref{sec:toricsystems}, and then deal with the process of augmentation in Section \ref{sec:augmentation}.

\subsection{Basics}
\label{sec:toricsystems}

In the following all surfaces are smooth and complete.
For general features of derived categories in algebraic geometry we refer to 
\cite{huybrechts}. By $\D^b(X)$ we denote the bounded derived category of 
coherent sheaves on some complex variety $X$.

\begin{definition}
        An object $E$ of $\D^b(X)$ is called \emph{exceptional} if it fulfills
\[
\Ext^i (E,E)=
\begin{cases}
\CC & \textnormal{if $i=0$}\\
0 & \textnormal{if $i\not=0$}.\\
\end{cases}
\]
An \emph{exceptional sequence} $\Ex$ is a finite sequence of exceptional objects $(E_1,\ldots,E_n)$
such that there are no morphisms back, that is, $\Ext^k (E_j,E_i)=0$ for $j>i$ and all $k$. Such a sequence is called \emph{strongly exceptional} if additionally $\Ext^k (E_j,E_i)=0$ for all $i,j$ and all $k>0$.
An exceptional sequence is called \emph{full} if $E_1,\ldots,E_n$ generate $\D^b(X)$, that is, the smallest full triangulated subcategory of $\D^b(X)$ containing all $E_i$'s is already $\D^b(X)$.
\end{definition}

We are primarily interested in full exceptional sequences of line bundles on rational surfaces. It is easy to see that the projection of the elements of such a sequence to $K_0(X)$ form a basis there, implying that any such sequence must have length equal to $\rk K_0(X)$. On the other hand, it is unknown whether exceptional sequences of this length are automatically full.
Unless explicitly stated otherwise, all exceptional sequences we consider will be  of length $\rk K_0(X)$ and consist only of line bundles. Furthermore, if $\phi\colon\tilde{X}\to X$ is some blowup and $\mathcal{L}$ is a line bundle on $X$, we will often use $\mathcal{L}$ to denote $\phi^*(\mathcal{L})$ as well, as long as the meaning is clear.

A related concept introduced by Hille and Perling \cite{hillperl08} are so-called \emph{toric systems}:
\begin{definition}
\label{def:toricsystem}
        A \emph{toric system}  on a rational surface $X$ of Picard number $n-2$ is a  sequence of divisor classes $\Ts=(A_1,\ldots,A_n)$ such that 
        \begin{enumerate}
                \item $A_i.A_{i+1}=1$;
                \item $A_i.A_j=0$ for $j\notin \{i-1,i,i+1\}$;
                \item $\sum_{i=1}^n A_i=-K_X$;
        \end{enumerate}
where we consider indices cyclically modulo $n$.
\end{definition}
One of the very nice ideas in \cite{hillperl08} is that from every 
toric system $\Ts=(A_1,\ldots,A_n)$ on a rational surface $X$ we can construct a toric surface $\tv(\Ts)$. Indeed, if we set $-b_i=\chi(\CO(A_i))-2$, then $\tv(\Ts)\coloneqq\tv(b_1,\ldots b_n)$ is a smooth toric surface. 
On the other hand, starting with some rational surface $X$ with an exceptional sequence $\Ex=(\CO(E_1),\ldots,\CO(E_n))$, we can construct an associated toric system $\Ts$, by setting
\begin{equation}\label{eqn:ts}
A_i \coloneqq E_{i+1}-E_i \textnormal{ and } A_n \coloneqq (E_1-K_X)-E_n.
\end{equation}
If a toric system $\Ts$ can be constructed in this manner, we call it \emph{exceptional}. If $\Ts$ can be constructed in this manner from a full/strongly exceptional sequence $\Ex$, we call it \emph{full/strongly exceptional}. Of course, a toric system $\Ts$ is (full/strongly) exceptional if and only if $(\CO,\CO(A_1),\CO(A_1+A_2),\ldots,\CO(\sum_{i=1}^{n-1} A_i))$ forms a (full/strongly) exceptional sequence; in such cases we call this the exceptional sequence associated to $\Ts$. Similarly, $\Ts$ is (strongly) exceptional if and only if 
\begin{equation*}
        H^l\left(X,\CO\left(\sum_{i=j}^k -A_i\right)\right)=0 \qquad \big(\textrm{and }H^{l+1}\left(X,\CO\left(\sum_{i=j}^k A_i\right)\right)=0\big)
\end{equation*}
for all $l\geq 0$ and $1\leq j \leq k<n$. Any cyclic permutation or reflection of the indices takes an (exceptional) toric system to an (exceptional) toric system and doesn't change the associated toric variety.\footnote{The property of being strongly exceptional is not invariant under cyclic permutation.} Finally, note that if $X$ is a toric surface, the invariant divisors $D_i$ circularly ordered form an exceptional toric system  $\Ts=(D_1,\ldots,D_n)$ which we call the \emph{canonical toric system}. In this case, we have $\tv(\Ts)=X$.

\subsection{Augmentation}
\label{sec:augmentation}

It follows from Proposition \ref{prop:blowupblowdown} that any homogeneous deformation of rational $\CC^*$-surfaces can be attained by repeatedly blowing up a deformation of a Hirzebruch surface.  For exceptional sequences, the situation is somewhat similar: In \cite{hillperl08}, a construction called \emph{augmentation} was established, which constructs new toric systems from blowups.  We  recall this notion,  adapting notation slightly.

\begin{definition}
\label{def:augmentation}
Let $\Ts=(A_1,\ldots,A_n)$ be a toric system on a rational surface $X$ and $\tilde{X} \rightarrow X$ a blowup
in one point with exceptional divisor $R$.
Then for every $1 \leq i \leq n$ we can construct  a toric system on $\tilde{X}$
\[
\Aug_i \Ts = (A_1,\ldots,A_{i-1},A_i-R,R,A_{i+1}-R,A_{i+2},\ldots,A_n)
\]
called the \emph{augmentation} of $\Ts$ at the position $i$. For any sequence of blowups, let $\Aug_{i_n,\ldots,i_1}\Ts$ denote the repeated augmentation of a toric system $\Ts$ at the positions $i_1,\ldots,i_n$.\footnote{Here we are only looking at what Hille and Perling call standard augmentations.}
\end{definition}

\begin{lemma}[{cf. \cite[Proposition 5.5]{hillperl08}}]\label{lemma:augex}
Let $\Ts$ be a toric system. Then $\Ts$ is (full) exceptional if and only if $\Aug_i \Ts$ is (full) exceptional.
\end{lemma}

\begin{proof}
The statement regarding exceptionality can easily reduced to the following question.
Let $X \rightarrow Y$  be a blowup of rational surfaces in a point with exceptional divisor $R$ and let $D$ denote an arbitrary divisor on $Y$ (or its pullback to $X$). The question now is whether the vanishing of all the cohomology of $\CO(D)$ is equivalent to the vanishing of all the cohomology of $\CO(D+R)$.
This can be shown by computing that $\chi(\CO(D))=\chi(\CO(D+R))$ by Riemann-Roch and using the exact sequence $0 \rightarrow \CO(D) \rightarrow \CO(D+R) \rightarrow \CO_R(D+R) \rightarrow 0$, see also \cite[Lemma 4.3]{hillperl08} for details.
In \cite[Proposition 5.5]{hillperl08}, it is also shown that the fullness of $\Ts$ implies the fullness of $\Aug_i \Ts$.
The argument there can easily be reversed.
\end{proof}

\begin{lemma}
\label{lemma:augintnum}
Let $\Ts$ be a toric system with $\tv (\Ts)=\tv(b_1,\ldots,b_n)$.
Then
\[
\tv(\Aug_i \Ts) = \tv(b_1,\ldots,b_{i-1},b_i+1,1,b_{i+1}+1,b_{i+2},\ldots,b_n).
\]
In other words, augmenting at position $i$ results in a blowup
of the associated toric variety $\tv(\Ts)$ at the same position (i.e. inserting
a new ray between the rays labeled with $i$ and $i+1$).
\end{lemma}

\begin{proof}
This statement can be shown by straightforward computation, using the formula
$\chi(\CO(D))=1+\frac{1}{2}(D^2-K.D)$ for the Euler characteristic.
\end{proof}

On the Hirzebruch surfaces $\F_r$ we will always choose the basis of $\Pic(\F_r)$ used in Example \ref{ex:hirzedefdiv}, that is, classes $\CO(P),\CO(Q)\in\Pic(\F_r)$ with
$P$ the divisor class of the fiber of ruling on $\F_r$ and $Q$ such that
$Q^2=r$ and $P\cdot Q=1$ (so $P$ and $Q$ are the generators of the nef cone). Note that for $\F_0$, $P$ and $Q$ are interchangeable.
Hille and Perling have calculated all possible toric systems on Hirzebruch surfaces:
\begin{prop}[{\cite[Proposition 5.2]{hillperl08}}]
\label{prop:exseqFr}
All toric systems on $\F_r$ up to cyclic permutation or reflection of the indices are of the form
\begin{align*}
        \Ts_{r,i}&=\big(P, iP+Q, P, -(r+i)P+Q\big);\textrm{ and}\\
        \tilde\Ts_{r,i}&=\big(-\frac{r}{2}P+Q,P+i(-\frac{r}{2}P+Q),-\frac{r}{2}P+Q,P-i(-\frac{r}{2}P+Q)\big)\textrm{ if $r$ is even}.
\end{align*}
$\Ts_{r,i}$ is always full exceptional, and strongly exceptional if and only if $i\geq 1$. $\tilde\Ts_{r,i}$ is exceptional only if $r=0$, or if $r=2$ and $i=0$, in which cases it is also full. Finally, $$\tv(\Ts_{r,i})=\F_{|r+2i|}\qquad\textrm{and}\qquad\tv(\tilde\Ts_{r,i})=\F_{|2i|}.$$
\end{prop}

Now let $X$ be any rational surface of Picard number $\rho\geq 2$. We call an exceptional toric system $\Ts$ on $X$ \emph{constructible} if there is some sequence of blowups $X=X^n\to\cdots\to X^0=\F_r$ such that $\Ts$ can be constructed inductively by augmenting some exceptional toric system on $\F_r$.
Note that by Lemma~\ref{lemma:augex}, a constructible toric system is automatically full.

\begin{prop}
\label{prop:toricsysgetsall}
Let $Y$ be a toric surface with the same Picard rank $\rho>2$ as a rational surface $X$.
Then there is an constructible toric system $\Ts$ on $X$ with $\tv(\Ts)=Y$.
\end{prop}

\begin{proof}
For $X$ and $Y$ there is a sequence of blowups reducing to
Hirzebruch surfaces $\F_r$ and $\F_s$ respectively,
say $X= X^n \rightarrow \cdots \rightarrow X^0=\F_r$ and
$Y= Y^n \rightarrow \cdots \rightarrow Y^0=\F_s$.

Assume that $r=s \mod 2$.
Let $\Ts$ be the toric system on $X$ attained by repeatedly augmenting $\Ts_{r,(s-r)/2}$ at the same positions where we blow up
$\F_s$ to get to $Y$.
Due to
Lemma~\ref{lemma:augintnum} it follows that $\tv(\Ts)=Y$.

Suppose instead that $r\not=s \mod 2$. Note that the blowdown $Y^1 \rightarrow \F_s$ isn't unique; there is also a blowdown $Y_1 \rightarrow \F_{s'}$ for either $s'=s+1$ or $s'=s-1$. Thus, by taking instead the blowdown  $Y_1 \rightarrow \F_{s'}$ we can in fact assume that $r=s \mod 2$.
\end{proof}

\section{Degenerations of Toric Systems}\label{sec:exdef}
We now consider the behaviour of exceptional sequences and toric systems under homogeneous deformations. In what follows, all surfaces will be rational and have a $\CC^*$-action, and we will only consider homogeneous deformations. 
We first present some general results in Section \ref{sec:exseqdef}. In Section \ref{sec:mutations}, we will relate degenerations of toric systems to mutations. Finally, in Section \ref{sec:compatibility}, we introduce the notion of \emph{compatibility} to analyze when a degeneration preserves the property of being exceptional.

\subsection{General Results}
\label{sec:exseqdef}
Consider any homogeneous deformation $\pi\colon \cX \rightarrow \base$ of rational $\CC^*$-surfaces. Let $\Ts=(A_1,\ldots,A_n)$ be any $n$-tuple of line bundles on a general fiber $\cX_s$. Then we define $\bpim(\Ts)$ to be the $n$-tuple $(\bpim(A_1),\ldots,\bpim(A_n))$, where we consider $\bpim$ as a map on divisor classes. We say that $\Ts$ \emph{degenerates} to $\bpim(\Ts)$, or equivalently that $\bpim(\Ts)$ \emph{deforms} to $\Ts$.

Our first observation is that degeneration preserves toric systems:
\begin{theorem}\label{thm:tsdeg}
	Let $\pi\colon \cX \rightarrow \base$ be a homogeneous deformation of rational $\CC^\ast$-surfaces and let $\Ts$ be a toric system on a general fiber $\cX_s$. Then $\bpim(\Ts)$ is a toric system, and $\tv(\Ts)=\tv(\bpim(\Ts))$.
Moreover, let $\pi'$ be
a blowup of this deformation as in Proposition \ref{prop:blowupblowdown}. Then
\begin{equation}
	\Aug_i \bpim (\Ts) = \bpimpr( \Aug_i \Ts).\label{eqn:augdeg}
\end{equation}
In other words, augmentation commutes with degeneration.
\end{theorem}
\begin{proof}
That $\bpim(\Ts)$ is again a toric system with $\tv(\bpim(\Ts)) = \tv(\Ts)$ is an immediate consequence of Proposition~\ref{prop:pim:preserves}.
Equation \eqref{eqn:augdeg} follows directly from Proposition ~\ref{prop:blowuppicommutes}.
\end{proof}

On the other hand, we can make a much stronger statement concerning the behavior or toric systems under deformation:
\begin{theorem}\label{thm:tsdef}
	Let $\pi\colon \cX \rightarrow \base$ be a homogeneous deformation of rational $\CC^\ast$-surfaces and let $\Ts$ be a toric system on the special fiber $\cX_0$. Then $\inv\bpim(\Ts)$ is a toric system on any general fiber $\cX_s$, and $\tv(\Ts)=\tv( \inv\bpim(\Ts))$. Furthermore, if $\Ts$ is exceptional/constructible/strongly exceptional, then so is $\inv\bpim(\Ts)$.
\end{theorem}
\begin{proof}
	The first two statements are shown exactly as in the proof of Theorem \ref{thm:tsdeg}. From Proposition \ref{prop:pim:preserves} we have that $\inv\bpim$  preserves the vanishing of cohomology, which implies that $\inv\bpim$ preserves (strong) exceptionality. Finally, if $\Ts$ is constructible, we can blow down $\cX_0$ to some $\cX_0'$ such that $\Ts=\Aug_i \Ts'$ for some constructible toric system $\Ts'$ on $\cX_0'$. We then blown down $\pi$ to $\pi'$ as in Proposition \ref{prop:blowupblowdown}, which induces a blowdown $\cX_s\to \cX_s'$. If  $\inv\bpimpr(\Ts')$ is constructible, then $\inv\bpim(\Ts)$ is as well, since by Theorem \ref{thm:tsdeg} we have that $$\inv\bpim(\Ts)=\Aug_i\inv\bpimpr(\Ts')$$
	with respect to the blowup $\cX_s\to \cX_s'$. The statement then follows by induction on the Picard number of $\cX_0$.
\end{proof}

Combining the above two theorems with Theorem \ref{thm:defcon} provides us with a proof of our Main Theorem \ref{mainthm:1} from the introduction.

\subsection{Mutations and Hirzebruch Surfaces}
\label{sec:mutations}
We now turn our attention to Hirzebruch surfaces, where we have some more explicit results. The first is the following proposition:
\begin{prop}\label{prop:hirzebruchdegen}
	Consider some homogeneous deformation $\pi\colon \cX \rightarrow \base$ with special fiber $\cX_0 = \F_{r+2\alpha}$ and general fiber $\cX_s = \F_r$ for $r>0$. Then for all $i$,
	$$\bpim(\Ts_{r,i})=\Ts_{r+2\alpha,i-\alpha}$$
	and if $r$ is even,
	$$\bpim(\tilde\Ts_{r,i})=\tilde\Ts_{r+2\alpha,i}.$$
In particular, exceptional toric systems degenerate to exceptional toric systems.
\end{prop}
\begin{proof}
We recorded all possible deformations in Example \ref{ex:hirzedefdiv}. As noted there, we have:
\[
\begin{array}{rcl}
	\bpim\colon \Pic(\F_r) &\rightarrow &\Pic(\F_{r+2\alpha})\\
\CO(P) & \mapsto & \CO(P)\\
\CO(Q) & \mapsto & \CO(Q-\alpha P)
\end{array}
\]
The proposition then follows from direct calculation.
\end{proof}

\begin{rem}\label{rem:hirz0}
	It might seem odd that in the above theorem, we must rule out the case $\cX_s=\F_0$. This is due to the interchangeable roles of $\CO(P)$ and $\CO(Q)$ in the basis of $\Pic(\F_0)$. In this case, either  $\bpim(\Ts_{0,i})=\Ts_{2\alpha,i-\alpha}$ as above or $\bpim(\Ts_{0,i})=\tilde\Ts_{2\alpha,i-\alpha}$, and either  $\bpim(\tilde\Ts_{0,i})=\tilde\Ts_{2\alpha,i}$ as above or $\bpim(\tilde\Ts_{0,i})=\Ts_{2\alpha,i-\alpha}$. In particular, the exceptional toric system $\Ts_{0,i}$ on $\F_0$ can be degenerated to $\tilde \Ts_{2\alpha, i-\alpha}$, which is not exceptional if $\alpha>1$. Thus, in general, the property of being exceptional is not preserved under degeneration.  We further discuss this in Section \ref{sec:compatibility}.
\end{rem}

We can further explain the above situation on Hirzebruch surfaces in terms of so-called \emph{mutations}. We first recall their definition from \cite{rudakov}:

\begin{definition}\label{def:mutation}
Let $(E,F)$ be a (not necessarily full) exceptional sequence of two arbitrary objects in $\D^b(X)$.
The \emph{left mutation} $L_F E$ of $F$ by $E$ is an object of $\D^b(X)$ that fits into the triangle
\[
L_E F \rightarrow \bigoplus_l \Hom(E,F[l]) \otimes E[-l] \overset{can}{\rightarrow} F \rightarrow L_E F[1].
\]
Similarly, we define the \emph{right mutation} $R_F E$ of $F$ by $E$ as the object that fits into the triangle
\[
E \overset{can^\ast}{\rightarrow} \bigoplus_l \Hom(E,F[l])^\ast[l] \otimes F  \rightarrow R_F E \rightarrow E[1].
\]
For an exceptional sequence $\Ex=(E_1,\ldots,E_n)$ of arbitrary objects we define the left mutation of $\Ex$ at
position $i$ as
\[
L_i \Ex=(E_1,\ldots,E_{i-1},L_{E_i} E_{i+1}, E_i, E_{i+2},\ldots,E_n),
\]
and analogously the right mutation of $\Ex$ at position $i$ is
\[
R_i \Ex=(E_1,\ldots,E_{i-1},E_{i+1}, R_{E_{i+1}} E_i, E_{i+2},\ldots,E_n).
\]
\end{definition}

\begin{rem}
In \cite{rudakov} it is shown that the left and right mutations of a (full) exceptional sequence are again (full) exceptional, and that the right and left mutation are inverses of each other.
\end{rem}
On Hirzebruch surfaces, it is possible to mutate an exceptional sequence of line bundles such
that the mutation  still consists of line bundles. 
According to the Proposition~\ref{prop:exseqFr} the exceptional sequences on $\F_r$, $r>0$ correspond to toric systems of the form $\Ts_{r,i}$ up to cyclic permutation or reflection of indices. Now let $\Ex$ be an exceptional sequence with toric system $\Ts_{r,i}$.
 Consider a mutation of $\Ex$ at the first position. To calculate this, we must look at
\[
\bigoplus_l \Hom(\CO,\CO(P)[l]) \otimes \CO[-l] \overset{can}{\rightarrow}\CO(P).
\]
Since
\[
\hom(\CO,\CO(P)[l])=h^l(X,\CO(P))=
\begin{cases}
2 & l=0\\
0 & l\not=0
\end{cases}
\]
and the map $can$ is surjective, $L_\CO \CO(P)$ is just the ordinary kernel of this map.
In fact, $L_\CO \CO(P)=\CO(-P)$, so
\[
L_1 \Ex = (\CO(-P),\CO, \CO((i+1)P+Q), \CO((i+2)P+Q)).
\]
Thus, on the level of toric systems, the left mutation of the toric system at the first position is
\[
L_1 \Ts_{r,i} = \Ts_{r,i+1}. \]
Since the first element of the mutated toric system
is again $P$, we can iterate this process. Hence, we denote
by $L_1^\alpha \Ts$ the result of left mutating $\Ts$ $\alpha$-times.
Note that we can extend this notion also to $\alpha \in \ZZ$.
Combining this with the previous proposition gives us for any deformation $\pi$ from $\F_{r+2\alpha}$ to $\F_r$
$$
\bpim(L_1^\alpha \Ts_{r,i})=\Ts_{r+2\alpha,i}.
$$
In particular, the $\alpha$-fold left mutation of the canonical toric system on $\F_r$ degenerates to the canonical toric system on $\F_{r+2\alpha}$. 
Likewise, changing to the viewpoint of deformation, we have
$$
\inv\bpim(L_1^{-\alpha} \Ts_{r+2\alpha,i})=\Ts_{r,i}.
$$
Although we originally ruled out the case that $r=0$, note that this isn't really necessary. We just need to choose the basis of $\CO(P),\CO(Q)\in\Pic(\F_0)$ such that $\bpim(\CO(P))=\CO(P)$.

We can extend the above discussion on Hirzebruch surfaces to general rational $\CC^*$-surfaces as follows: 
\begin{theorem}
Consider a homogeneous deformation or degeneration of rational $\CC^\ast$-surfaces from $X$ to a toric surface $Y$, which blows down to a deformation respectively degeneration from $\F_r$ to $\F_{r+2\alpha}$ for $\alpha\in \ZZ$. Consider the augmentation with respect to this blowdown $\Ts_Y=\Aug_{i_n,\ldots,i_1} \Ts_{r+2\alpha,0}$ such that $\Ts_Y$ is the canonical toric 
system on $Y$.\footnote{If either $r=0$ or $r+2\alpha=0$, we must choose the basis $P,Q$ of $\Pic(\F_0)$ as above.} Then the toric system 
\[
\Ts = \Aug_{i_n,\ldots,i_1} L_1^\alpha \Ts_{r,0}
\]
on $X$ deforms respectively degenerates to $\Ts_Y$. 
\end{theorem}

\begin{proof}
	Combine the case for Hirzebruch surfaces discussed above with Equation \eqref{eqn:augdeg} from Theorem \ref{thm:tsdeg}.
\end{proof}

\subsection{Compatibility}
\label{sec:compatibility}
Although one might hope that homogeneous degenerations preserve exceptional toric systems, we have seen in Remark \ref{rem:hirz0} that this is in general not the case.
Likewise, the property of being strongly exceptional is also not preserved.
Indeed, consider the strongly exceptional toric system $\Ts_{r,i}$ on $\F_r$, where $i\geq 1$. As we saw above, this can be degenerated to $\Ts_{r+2\alpha,i-\alpha}$ on $\F_{r+2\alpha}$, which is no longer strongly exceptional if $i<\alpha+1$.

However, the situation isn't hopeless---for any degeneration, we can identify a subset of exceptional toric systems which degenerate to exceptional toric systems:
\begin{definition}
\label{def:compatible}
	Let $\pi \colon \cX \rightarrow \base$ be a homogeneous deformation of $\CC^*$-surfaces, and let $\Ts$ be a constructible toric system on a general fiber $\cX_s$. We say that $\Ts$ is \emph{compatible} with $\pi$ if:
	\begin{enumerate}
		\item $\cX_s$ is a Hirzebruch surface and $\bpim(\Ts)$ is exceptional; or
		\item There is a blowdown of $\pi$ to $\pi'$ inducing a blowdown $\cX_s\to \cX_s'$ such that $\Ts$ is an augmentation of a toric system $\Ts'$ on $\cX_s'$ compatible with $\pi'$.
	\end{enumerate}
\end{definition}

Proposition \ref{prop:hirzebruchdegen} and Remark \ref{rem:hirz0} thus give us an explicit description of the toric systems compatible with any deformation of Hirzebruch surfaces. The second condition above then can be applied inductively to determine all toric systems compatible with a given deformation. The importance of compatibility is made clear by our Main Theorem \ref{mainthm:2}, which we restate here:
\begin{theorem}
\label{thm:compat:constr}
	Let $\pi$ be a homogeneous deformation of rational $\CC^*$-surfaces with general fiber $\cX_s$ and let $\Ts$ be a constructible toric system on $\cX_s$. Then $\Ts$ is compatible with $\pi$ if and only if $\bpim( \Ts)$ is a constructible toric system. In particular, if $\Ts$ is compatible with $\pi$, then $\bpim (\Ts)$ is a full exceptional toric system.
\end{theorem}

\begin{proof}
	We first prove that if $\Ts$ is compatible with $\pi$, then $\bpim(\Ts)$ is constructible; this is done by induction on the Picard number $\rho$ of $\cX_0$. The case $\rho=2$ follows directly from the definition of compatibility. On the other hand, the induction step follows from Equation \eqref{eqn:augdeg} and Lemma \ref{lemma:augex}.

	Now suppose that $\Ts$ isn't compatible with $\pi$, but $\Ts_{\cX_0}\coloneqq\bpim(\Ts)$ is constructible. Then there is a blowdown $\cX_0\to \cX_0'$ and a constructible toric system $\Ts_{\cX_0'}$ on $\cX_0'$ such that $\Ts_{\cX_0}$ is an augmentation of $\Ts_{\cX_0'}$ with respect to this blowup. The blowdown $\cX_0\to \cX_0'$ induces a unique blowdown of $\pi$ to some $\pi'$ with special fiber $\cX_0'$ and some general fiber $\cX_s'$. If we set $\Ts'\coloneqq\inv\bpimpr(\Ts_{\cX_0'})$, then by Equation \eqref{eqn:augdeg} we can conclude that $\Ts$ is an augmentation of $\Ts'$, and that if $\Ts'$ is compatible with $\pi'$, then $\Ts$ must be compatible with $\pi$. Applying induction, we arrive at a contradiction, and can thus conclude that $\Ts_{\cX_0}$ must not have been constructible.
\end{proof}

\begin{rem} 
Let $\pi \colon \cX \to \base$ be a homogeneous deformation of rational $\CC^\ast$-surfaces.
Assume that all exceptional toric systems on the special fiber $\cX_{0}$ are constructible.
It is an immediate consequence of the Theorems \ref{thm:tsdef} and \ref{thm:compat:constr},
that a toric system $\Ts$ on $\cX_0$ is compatible with $\pi$ if and only if $\bpim(\Ts)$ is exceptional.
However, we will see in Section~\ref{sec:constr:ts} that the assumption here isn't fulfilled in general.
\end{rem}

\begin{rem}
It is not difficult to find constructible toric systems which are not compatible with certain deformations. Indeed, consider any toric surface $\cX_s$ with multiple invariant minus one curves, and let $\Ts$ be a constructible toric system on $\cX_s$ such that $\tv(\Ts)$ only has a single invariant minus one curve. Then we claim there is a degeneration of $\cX_s$ with which $\Ts$ is not compatible. Indeed, since $\tv(\Ts)$ only has a single invariant minus one curve, there exists a unique blowdown  $\cX_s\to {\cX}_{s}'$  such that $\Ts$ is the augmentation of a constructible toric system on ${\cX}_s'$; let $C$ be the corresponding minus one curve. Now let $\pi$ be any degeneration of $\cX_s$ to some $\cX_0$ such that the vertex $v$ corresponding to $C$ has at least degree $2$ in the corresponding degeneration diagram. Then $\Ts$ is not compatible with $\pi$, since $\pi$ cannot be blown down to have general fiber $\cX_s'$.
\end{rem}

\section{Constructible Toric Systems}
\label{sec:constr:ts}
We now will concentrate on the constructibility of exceptional toric systems on toric surfaces of Picard rank 3 and 4 and prove Main Theorem \ref{mainthm:3}. 
To begin, we consider constructible toric systems on the del Pezzo surfaces of degrees 6 and 7. We then discuss how to connect toric surfaces of Picard rank 3 and 4 via homogeneous degeneration to these surfaces.  To complete the proof, we analyze the behaviour of constructible toric systems under the corresponding maps of Picard groups. We also present an example of an exceptional toric system on a toric surface of rank 5 which is not constructible.

\subsection{Automorphisms and $\dP_6$ and $\dP_7$}\label{sec:auto}
To start out with, note that
the relationship between a toric system $\Ts$ and the associated toric surface $\tv(\Ts)$ is
actually an instance of Gale duality, as observed in \cite{hillperl08}. More precisely, for a
toric system $\Ts = (A_1, \ldots, A_n)$ on a rational surface $X$, there is the exact sequence
\[
\begin{array}{ccccccccc}
0 & \longrightarrow &  \Pic(X) & \longrightarrow & \ZZ^n & \longrightarrow & \coker \cong \ZZ^2 & \longrightarrow & 0 \\
&&  \CO(D) & \mapsto & (D.A_1, \ldots,D.A_n)
\end{array}
\]
The Gale-dual configuration of $\Ts$ is the image of the canonical basis of $\ZZ^n$ in $\ZZ^2$, which are just the rays of $\tv(\Ts)$.
Conversely, given the rays of a toric surface, we obtain a toric system by the dual procedure.
Gale-dual configurations are unique, but only up to isomorphism. This can be reformulated as the following proposition.

\begin{prop}
\label{prop:isometry}
Let $\Ts$ and $\Ts'$ be two toric systems on a rational surface $X$ with $A_i^2 = (A_i')^2$, and consequently $\tv(\Ts)=\tv(\Ts')$.
Then there is an automorphism $\phi \in \Pic(X)$ with $\phi(\Ts) = \Ts'$ which preserves the intersection pairing and the canonical divisor. 
\end{prop}

\begin{rem}
\label{rem:nicebasis}
Such automorphisms are rare. For the Picard group of del Pezzo surfaces, there are only finitely many such automorphisms, which were studied in \cite[Chapter 4]{manin}.
They may be obtained as follows. On any del Pezzo surface $X \not= \PP^1 \times \PP^1$, there exist special bases $(\CO(H),\CO(R_1),\ldots,\CO(R_n))$ of the Picard group of $X$ which diagonalize the intersection pairing with signature $(1,-1,\ldots,-1)$ and such that the canonical divisor class equals $-3H+\sum R_i$.
The base transformation between two such bases are the mentioned automorphisms.
One such basis may be constructed explicitly as follows.
Indeed, the del Pezzo surface $X$ can be obtained from $\PP^2$ by a sequence of blowup in generic points. Let $\CO(H)$ be
the pull-back of $\CO_{\PP^2}(1)$ and $R_i$ the pull-backs of the exceptional divisors to $X$.
Then $(\CO(H),\CO(R_1), \ldots, \CO(R_{n}))$  
is a basis of the desired type.

We note that such special bases exist for all toric surfaces with the exception of even Hirzebruch surfaces. Indeed, any toric surface $X$ of Picard rank greater than two can be blown down to at least one Hirzebruch surface $\F_{2a+1}$.
There we can choose the basis $(\CO(H)=\CO(Q-aP), \CO(R_1) = \CO(Q-(a+1)P))$ of $\Pic(\F_{2a+1})$ which has the same properties as in the del Pezzo case.
Going back to $X$ by pulling back the divisors and adding the exceptional ones, we get also such a special basis $(\CO(H),\CO(R_1),\ldots,\CO(R_n))$ of $\Pic(X)$.
\end{rem}

We now focus on the del Pezzo surfaces $\dP_7$ and $\dP_6$ of respective degrees $7$ and $6$. These are in fact toric, with $\dP_7=\tv(1,1,1,0,0)$ and $\dP_6=\tv(1,1,1,1,1,1)$.
The above automorphism groups are respectively isomorphic to $\ZZ_2$ and the Weyl group $W(A_1\times A_2)$, which consists of $12$ elements.

\begin{prop}\label{prop:dpconstr}
On $\dP_7$ and $\dP_6$ any toric system is constructible.
\end{prop}

\begin{proof}
We exhibit $2$, respectively $12$, different constructible toric systems on $\dP_7$ and $\dP_6$ yielding the same toric variety. These must be all toric systems by the above discussion.

There is a unique blowdown of $\dP_7$ to $\PP^1\times\PP^1$. 
By augmentation at some fixed position, the two differing families of exceptional toric systems on $\PP^1\times\PP^1$ yield two differing families of constructible toric systems on $\dP_7$.
There are six ways to blow down $\dP_6$, all resulting in $\dP_7$. 
Via augmentation, each such blowdown yields two differing families of constructible toric systems on $\dP_6$ for a total of $12$.
\end{proof}

We now restate our Main Theorem \ref{mainthm:3}:
\begin{theorem}
\label{thm:constr:34}
Let $X$ be a toric surface of Picard rank $3$ or $4$. Then any exceptional toric system on $X$ is constructible.
\end{theorem}
Let us briefly explain the idea of the proof. We first show that
$\dP_7$ homogeneously degenerates directly to 
any Picard rank $3$ toric surface, and that $\dP_6$ homogeneously degenerates in a finite number of steps to any Picard rank $4$ toric surface. We then use the induced isomorphisms of Picard groups to compare toric systems. By Theorem \ref{thm:tsdef} and Proposition \ref{prop:dpconstr}, we know that all exceptional toric systems on any rank 3 or 4 surface come from constructible toric systems on $\dP_7$ or $\dP_6$. Thus, we must check for each degeneration we consider that, for any constructible toric system degenerating to something nonconstructible, the degenerated toric system is no longer exceptional. For all but finitely of our degenerations, this will follow by considering blowdowns of the degeneration in question. The remaining finitely many cases may then be dealt with by hand.

\subsection{Homogeneous Deformations to $\dP_6$ and $\dP_7$}

From Lemma \ref{lemma:dckn}, we have that rank 3 toric surfaces are exactly those of the form
\begin{align*}
X_r=\tv(r,1,1,-r+1,0) \qquad r\geq 0.
\end{align*}
 Note that $X_0=X_1=\dP_7$. By considering all possible invariant blowups of these surfaces, one easily confirms that rank 4 toric surfaces are exactly those of the form
\begin{align*}
X_r^A&=\tv(r+1,1,1,-r+1,1,1)\qquad &r\geq 0;\\
X_r^B&=\tv(r+1,1,2,1,-r+1,0)\qquad &r\geq 0;\\
X_r^C&=\tv(r+1,2,1,2,-r,0)\qquad &r\geq 0.
\end{align*}
Here, we have $X_0^A=\dP_6$.

These surfaces are all related by homogeneous deformations as depicted in Figure \ref{fig:defgraphfour}. Here, arrows of the form $\baddefto$ and $\defto$ represent homogeneous deformations. All pictured deformations of the form $\defto$ can be constructed as in Theorem~\ref{thm:toricdef}. For Picard rank $3$, the deformations of the form $\baddefto$ are depicted in Figure~\ref{fig:tildehirze}. We leave it to the curious reader to find degeneration diagrams giving the remaining such deformations.

\begin{lemma}\label{lemma:goodcases}
Let $\pi\colon\cX\to\base$ be a homogeneous deformation giving rise to an arrow of the form $\defto$, and $\Ts$ a constructible toric system on $\cX_s$. Then $\bpim(\Ts)$ is constructible.
\end{lemma}
\begin{proof}
	We first observe that for any of the given deformations, $\cX_s$ and $\cX_0$ have the same number of exceptional divisors. Thus, any blowdown of the \emph{general} fiber $\cX_s$ also induces a blowdown of $\pi$. 
	
	Since $\Ts$ is constructible, there is some blowdown $\cX_s'$ of the general fiber with a constructible toric system $\Ts'$ such that $\Ts$ arises from $\Ts'$ via augmentation. Let $\pi'$ be the corresponding blowdown of $\pi$. Now suppose that the Picard rank of $\cX_s$ is 3. Then $\cX_s'$ is some Hirzebruch surface $\F_r$ for $r\geq 1$, and by Proposition \ref{prop:hirzebruchdegen}, $\overline{(\pi')}^\circ(\Ts')$ is constructible. It follows from Theorem \ref{thm:tsdeg} that $\bpim(\Ts)$ is also constructible.

	Suppose instead that the Picard rank of $\cX_s$ is 4. A straightforward calculation shows that the special and general fibers of $\pi'$ both have two exceptional divisors, and that $\cX_s'$ doesn't blow down to $\F_0$. Blowing down $\pi'$ again and applying Theorem \ref{thm:tsdeg} twice then leads to the desired conclusion.
\end{proof}

\begin{figure}[htbp]
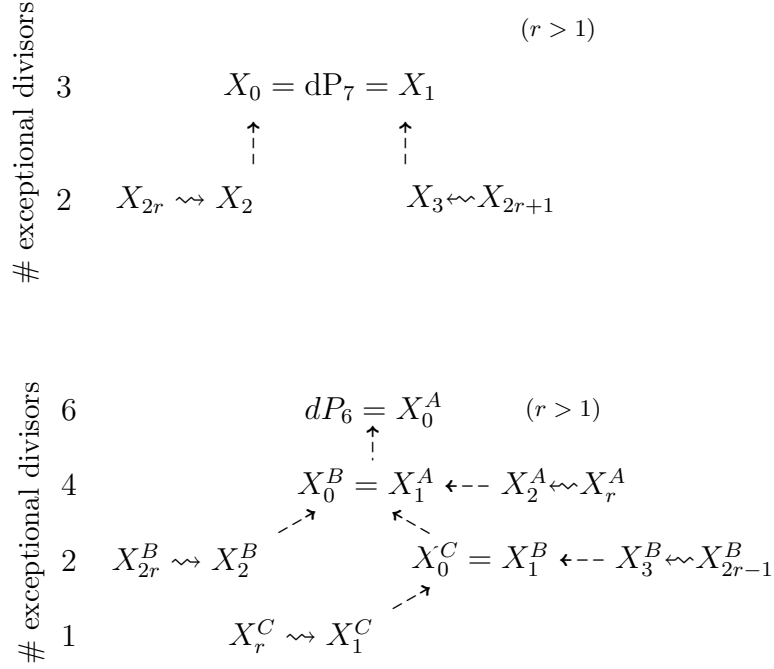

        \begin{center}
\defgraphthree
\vspace{1cm}

\defgraphfour
        \end{center}
        \caption{Deforming homogeneously to $\dP_7$ and $\dP_6$.}\label{fig:defgraphfour}
\end{figure}

\subsection{Remaining Cases}
As a consequence of Lemma \ref{lemma:goodcases}, we now only need to check that for arrows of the form $\baddefto$, the corresponding map of Picard groups maps constructible toric systems to either constructible toric systems, or non-exceptional toric systems. We show how to do this completely for the Picard rank 3 case, and present some of the cases for Picard rank 4. The curious reader may find the remaining cases carried out in detail in \cite{hochen}.

\vspace{.1cm}\noindent{\framebox{$X_2\baddefto \dP_7$ and $X_3\baddefto \dP_7$}}\hspace{.2cm}
We  now focus on the toric systems on $\dP_7$.
We can express the exceptional toric system $\Ts_i = \Ts_{1,i}$ on $\F_1$ from Proposition~\ref{prop:exseqFr} 
 in the special basis $(\CO(H),\CO(R_1))$ of $\Pic(\F_1)$ from Remark \ref{rem:nicebasis}  as follows (recall that $P = H-R_1, Q = H$):
\[
\Ts_i = (H-R_1, (i+1)H-iR_1, H-R_1, -iH+(i+1)R_1)
\]
Augmenting $\Ts_i$ to get a toric system on $\dP_7$ doesn't really depend on the position, since the results only differ by cyclic permutation or reflection, so we can write
\begin{equation}
\label{eq:tildeTS}
\tilde \Ts_i = \Aug \Ts_i = (H-R_1-R_2,R_2,(i+1)H-iR_1-R_2, H-R_1, -iH+(i+1)R_1)
\end{equation}
The only automorphism of $\Pic(\F_1)$ preserving the intersection pairing and canonical class comes from swapping $R_1$ and $R_2$.
So there are two different families of exceptional toric systems $\tilde \Ts_i$ and $\tilde \Ts^{\text{swap}}_i$ on $\dP_7$ up to permutation and reflection. Note that $\tilde \Ts_0$ and $\tilde \Ts^{\text{swap}}_0$ only differ by cyclic permutation and a reflection.

When deforming from $X_2$ or $X_3$ to $\dP_7$, only one of these two exceptional divisors $R_i$ on $\dP_7$ is mapped to an exceptional divisor on $X_r$ by $\bpim$, say $R_1$ is that divisor. So $\tilde \Ts^{\text{swap}}_i$ is automatically mapped by $\bpim$ to a constructible toric system on $X_2$ or $X_3$, respectively.
Since the image of $\tilde \Ts_i$ under $\bpim$ isn't constructible for $i\not=0$, we must show that it is also not exceptional.
Fix now the basis $(\CO(H),\CO(R_1),\CO(R_2))$ as in Remark~\ref{rem:nicebasis} such that $\bpim(R_2)$ is not an exceptional divisor.

Let $i$ be an non-zero integer.
Since $\tilde \Ts_i = (A_1,\ldots,A_5)$ is exceptional, $H^\bullet(\CO(\sum_{l=j}^k -A_l))$ vanishes for $1\leq j\leq k<5$. In particular, the cohomology of 
\[
E^{(i)}_1 = \CO(- A_3) = \CO(-(i+1)H+iR_1+R_2\]
and
\[E^{(i)}_2 = \CO\left(\sum_{l=2}^4 -A_l\right) = \CO(-(i+2)H+(i+1)R_1)
\]
vanishes. 

The homogeneous deformation $X_2 \baddefto X_0 = \dP_7$ and $X_3 \baddefto X_1 = \dP_7$ are depicted in Figure~\ref{fig:tildehirze}. We can compute the corresponding basis $(\CO(H),\CO(R_1),\CO(R_2))$ by the procedure from Remark~\ref{rem:nicebasis}, and moreover, the image of this basis and $E^{(i)}_1$ under $\bpim$. 
The results are noted in Table~\ref{tab:tildehirze}, where we have written divisor classes for readability.
We see that $\bpim(E^{(i)}_1)$ is effective for $i<0$, so that $H^0(\bpim(E^{(i)}_1))\not=0$.
It then follows immediately that the Serre-dual of $\bpim(E^{(i)}_2)$ is effective for $i>0$, 
since 
\[
\CO(K)\otimes (E^{(i)}_2)^{-1} = \CO(-3H+R_1+R_2\ +\ (i+2)H-(i+1)R_1) = \CO((i-1)H-iR_1+R_2)=E^{(-i)}_1\!.
\]
In any case, $\bpim(\tilde \Ts_i)$ is not exceptional for $i \not=0$.
This completes the proof that on a toric surface $X$ of Picard rank $3$, any exceptional toric system is constructible.

\begin{figure}[htbp]
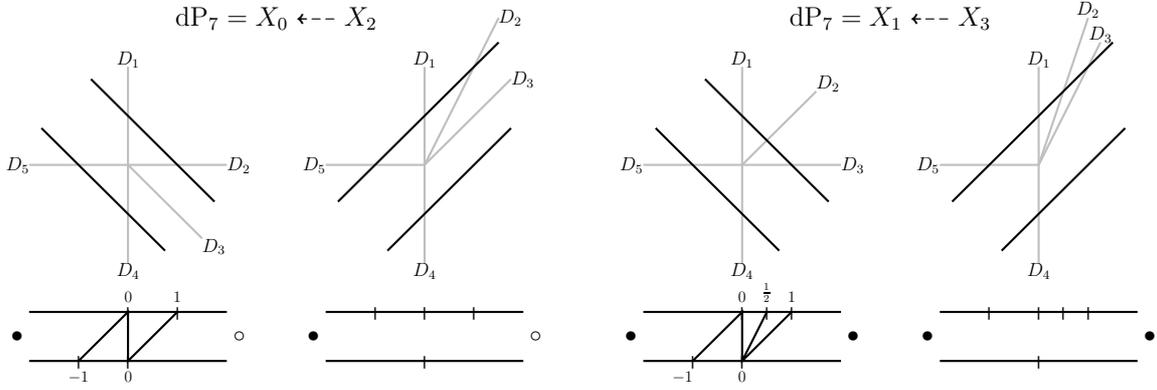

        \begin{center}
\figDPthreeA
\figDPthreeB
        \end{center}
	\caption{Homogeneous deformations from $X_2$ and $X_3$, respectively, to $\dP_7$.}\label{fig:tildehirze}
\end{figure}

\begin{table}[htbp]
\begin{tabular}{l|l|l}
& $X_2 \baddefto X_0$ & $X_3 \baddefto X_1$\\
\hline
$(H,R_1,R_2)$ & $(D_4+D_5,D_2,D_4)$ & $(D_1+D_5,D_3,D_1)$\\
 & \multicolumn{1}{|r|}{(using: $ \TVdPtB \rightarrow \TVFoneB \rightarrow \TVPtwoB $ )} &
\multicolumn{1}{|r}{(using:  $ \TVdPtA  \rightarrow \TVFoneA \rightarrow \TVPtwoA  $)} \\
$\bpim(H,R_1,R_2)$ & $(D_1+D_2+D_5,D_2,D_1+D_2)$ & $(D_1+2D_5,D_3,D_1+D_5)$ \\
$\bpim(E^{(i)}_1)$ & $ -iD_1-(i+1)D_5 $ & $-i(D_1+2D_5)+iD_3$ \\
 & & \multicolumn{1}{|r}{$\sim -i(D_1+2D_2+D_3)$} \\
 & & \multicolumn{1}{|r}{ (using: $ D_5 \sim D_2+D_3$)}\\
\end{tabular}
\caption{}\label{tab:tildehirze}
\end{table}

\vspace{.1cm}\noindent{\framebox{$X^B_3 \baddefto X^B_1$ and $X^A_2 \baddefto X^A_1$}}\hspace{.2cm}
Let $\Ts$ be a constructible toric system on the general fiber $\cX_s$, which is $X^B_1$ or $X^A_1$, respectively. Since the general fiber has the same number of exceptional divisors as the special one $\cX_0$, we can argue in the same way as in the proof of Lemma~\ref{lemma:goodcases} that $\bpim(\Ts) = \Aug_i \bpimpr(\Ts')$ for some $i$ and some constructible toric system $\Ts'$ on a blowdown $\cX_s'$ of the general fiber.
Since in the previous case we have shown the statement of Theorem \ref{thm:constr:34} already for Picard rank $3$, $\bpimpr(\Ts')$ is either constructible or not exceptional. By Lemma~\ref{lemma:augex} this also holds for $\bpim(\Ts)$.

\vspace{.1cm}\noindent{\framebox{$X^C_1 \baddefto X^C_0$}}\hspace{.2cm}
Since the image of the exceptional divisor $D_1$ on $X_0^C$ isn't exceptional on $X_1^C$,
the corresponding blowdown $X_0^C = \tv(1,2,1,2,0,0) \rightarrow \tv(1,1,2,0,-1) \cong X_2$ cannot be extended to a blowdown of the deformation $X^C_1 \baddefto X^C_0$.
We thus have to look at the augmentations of the exceptional toric systems on $X_2$. By the above discussion, they are of the form 
$\tilde\Ts_i\coloneqq\pim(\tilde\Ts_i^{\text{swap}})$ with 
$\tilde\Ts_i^{\text{swap}}$ as above an exceptional toric system on $\dP_7$ and $\pi$ the deformation $X_2 \baddefto \dP_7$.
Consider the special basis $(\CO(H),\CO(R_1),\CO(R_2))$ of $\Pic(X_2)$ from Remark~\ref{rem:nicebasis}.
By pulling back and setting $R_3=D_1$ we get the basis $(\CO(H),\CO(R_1),\CO(R_2),\CO(R_3))$ of $\Pic(X_0^C)$.
With respect to the basis $(\CO(H),\CO(R_1),\CO(R_2))$, note that 
 $\tilde\Ts_i$ on $X_2$ is as in Equation~\eqref{eq:tildeTS}.
Now, the augmentation of $\tilde\Ts_i$ depends on the position.
In the Table \ref{tab:picfour}, we record all these augmentations $\hat\Ts^j = \Aug_j \tilde\Ts_i$ and their behaviour under the map $\bpim \colon \Pic(X_0^C) \rightarrow \Pic(X_1^C)$.
More precisely, we discuss whether the elements $\bpim(\CO(\sum_{k=l}^m -A_k))$ give rise to an exceptional sequence, i.e. have no cohomology.
Here we write $\hat\Ts^j=(\hat A^j_1, \ldots, \hat A^j_6)$.
As is noted in the table, the only cases which are exceptional are in fact those which are constructible.

\begin{table}
\begin{tabular}{c|l}
$\hat\Ts^1$ & $\pim(-\hat A^1_3)$ effective. \\
$\hat\Ts^2$ & $\pim(-\hat A^2_2)$ effective. \\
$\hat\Ts^3$ & $\pim(-\hat A^3_3)$ effective for $i<0$, \\
& Serre dual of $\pim(-\hat A^3_1 - \ldots - \hat A^3_4)$ effective for $i>0$, \\
& constructible for $i=0$.\\
$\hat\Ts^4$ & $\pim(-\hat A^4_3)$ effective for $i<-1$,\\
& Serre dual of $\pim(-\hat A^4_1 - \ldots - \hat A^4_5)$ effective for $i>-1$,\\
& constructible for $i=-1$.\\
$\hat\Ts^5$ & $\pim(-\hat A^5_3)$ effective for $i<-1$,\\
& $\pim(-\hat A^5_5)$ effective for $i>-1$,\\
& constructible for $i=-1$.\\
\end{tabular}
\caption{}\label{tab:picfour}
\end{table}

More details for this case and for the similar remaining cases may be found in \cite{hochen}.

\subsection{Toric Surfaces of Higher Picard Rank}
For toric surfaces of higher Picard rank, not all exceptional toric systems are constructible. Indeed, we may consider the following example. 

\begin{ex}[A nonconstructible exceptional toric system]
\label{ex:nonconstr:5}
Let $X$ be the toric surface $\tv(2,1,1,1,1,2,1)$. A straightforward calculation shows that the toric system
\begin{multline*}
\Ts = (- D_4+D_7, D_3 + D_4 -D_7, - D_3+D_7, D_2 + 4 D_3 + 3 D_4-D_7, \\
-D_2 + D_4 + D_5, D_2 + D_3 - D_5, - 3 D_3 - 2 D_4 + D_5-D_7)
\end{multline*}
is exceptional, where $D_i$ is the invariant divisor with self-intersection $-b_i$. On the surface $X$, the exceptional divisors are $D_2,D_3,D_4,D_5$ and $D_7$, and they form a basis for the Picard group. Since the elements of $\Ts$ are already linear combinations of these divisors, none can be an exceptional divisor. Thus,
we see that $\Ts$ can't be constructible.
We want to stress that we can't apply Lemma~\ref{lemma:augex} here to deduce fullness.
In the subsequent article \cite{hochen:11} by the first author, the fullness is established since the example above differs from a constructible toric system only by an autoequivalence of the derived category (namely a spherical twist).
\end{ex}

\section{Noncommutative Deformations}
\label{sec:noncomdef}
We now demonstrate how the degenerations of exceptional sequences we consider can be used to construct deformations of derived categories.
If $\Ex=(E_1,\ldots,E_n)$ is a full strongly exceptional sequence (not necessarily of line bundles) on a variety $X$, then the corresponding \emph{tilting sheaf} is defined to be 
$$
\tilt=\bigoplus E_i.
$$
There is then an equivalence of categories between $\D^b(X)$ and $\D^b(\EN(\tilt)\Mod)$, see \cite{MR992977}. The algebra $\EN(\tilt)$ can be described as a finite path algebra with relations, see \cite{perling09} for examples.  Note that if $\Ex$ is a full strongly exceptional sequence of line bundles on a rational surface $X$,  then we have
$$
\EN(\tilt)=\bigoplus_{j\leq k<n} H^0\Big(X,\CO\Big(\sum_{i=j}^k A_i\Big)\Big)
$$
where $\Ts$ is the toric system corresponding to $\Ex$.

Let $\Gamma$ be a family of algebras parametrized over a base variety $\base$ such that for every $t\in \base$, the corresponding algebra $\Gamma_t$ has the form $\End(\T)$ for some tilting sheaf $\T$ on some variety. We loosely call the family $\Gamma$ a \emph{noncommutative deformation} as it offers a way of ``deforming'' varieties via derived categories. Several concrete examples are presented in Section $7$ of \cite{perling09}.
Now suppose that $\pi\colon\cX \to \base$ is a homogeneous deformation of a rational $\CC^*$-surface $\cX_0$ and $\Ex$ is a full strongly exceptional sequence of lines bundles on $\cX_s$ which degenerates to a full strongly exceptional sequence. This data naturally gives rise to a noncommutative deformation.
Indeed, representing the elements of the corresponding toric system $\Ts$ by divisors $A_i$,
$$\Gamma_t=\bigoplus_{j\leq k<n} H^0\Big(\cX_t,\CO\Big(\sum_{i=j}^k (A_i^\tot)_{|{\cX_t}}\Big)\Big)$$
gives a family of endomorphism algebras with $\Gamma_0$ and $\Gamma_s$ describing the derived categories of the special and general fibers $\cX_0$ and $\cX_s$.

We wish to describe such families explicitly in the case of Hirzebruch surfaces by defining a family of quivers.
Fix some $r\geq 0$ and $i>0$, and $0<\alpha< i$. We then have the deformation $\pi$ from $\F_{r+2\alpha}$ to $\F_r$ given by the degeneration diagram $(\M(r,\alpha),G)$ from Example \ref{ex:hirzedefdiv}. Furthermore,
the toric system $\Ts_{r,i}$ is a strongly exceptional toric system on $\F_r$, which degenerates to the strongly exceptional toric system $\Ts=\Ts_{r+2\alpha,i-\alpha}$ on $\F_{r+2\alpha}$.
To calculate $\Gamma_t$ for $t \in \base$, we thus need to know the cohomology groups $H^0(\F_r,\CO(iP+Q))$, $H^0(\F_r,\CO(P))$, $H^0(\F_{r+2\alpha},\CO((i-\alpha)P+Q))$, and $H^0(\F_{r+2\alpha},\CO(P))$. These can be calculated using standard toric methods. We will represent $P$ and $Q$ as divisors on $\F_r$ and $\F_{r+2\alpha}$ as we have in Example \ref{ex:hirzedefdiv}. That is, on $\F_r$ we represent $P$ and $Q$ by respectively $D_{t,0}$ and $D_{t,1/\alpha}$, and on $\F_{r+2\alpha}$ we represent $P$ and $Q$ by respectively $D_{t,1/\alpha}$ and $(r+\alpha)D_{t,-1/(r+\alpha)}+D_{0,0}+\alpha D_{0,1/\alpha}$. In Figure \ref{fig:gs}, we present polytopes where each lattice point corresponds to a monomial element of the basis of the relevant cohomology group. Note that for $t\neq 0$, instead of having monomials in the usual variables $x$ and $y$ we have monomials in the variables $x$ and $\frac{y}{y-t}$.

\begin{figure}[htbp]
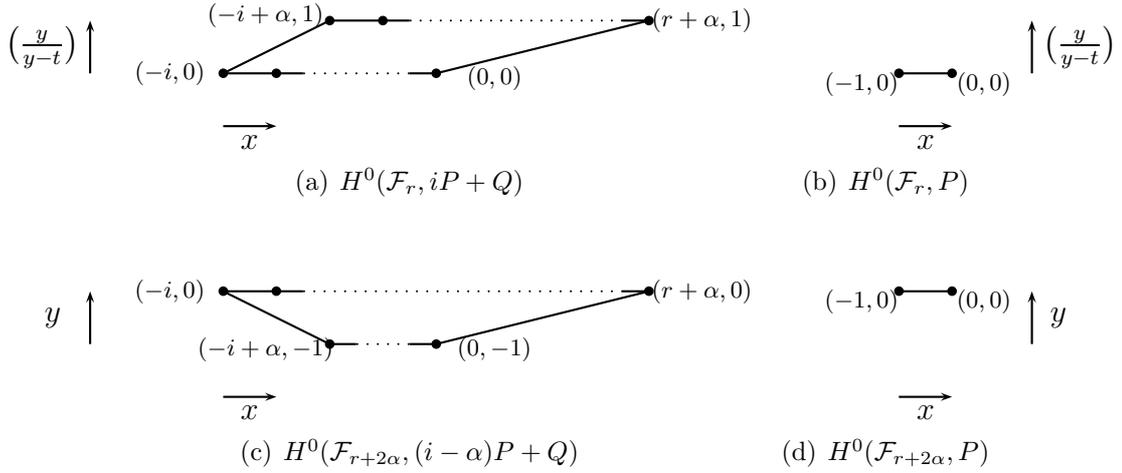

	\begin{center}\subfigure[$H^0(\F_r,iP+Q)$]{\tsgsi}
	\subfigure[$H^0(\F_r,P)$]{\tsgsiii}
		\subfigure[$H^0(\F_{r+2\alpha},(i-\alpha)P+Q)$]{\tsgsii}
		\subfigure[$H^0(\F_{r+2\alpha},P)$]{\tsgsiiii}
	\end{center}
	\caption{Global Sections of Toric Systems on $\F_r$ and $\F_{r+2\alpha}$}\label{fig:gs}
\end{figure}

We first concern ourselves with the global sections of $\CO(iP+Q)$ and $\CO((i-\alpha)P+Q)$ on $\F_r$ and $\F_{r+2\alpha}$, respectively. For $t\in \base$, let $b_{j}^{(t)}=x^j$ and let $d_{j}^{(t)}=x^j \frac{y}{y-t} $. Note that for $t\neq 0$,  we have that $b_{-i}^{(t)},\ldots,b_0^{(t)},d_{-i+\alpha}^{(t)},\ldots,d_{r+\alpha}^{(t)}$ is a basis for $H^0(\F_r,\CO(iP+Q))$. However, for $t=0$ we run into difficulties, since  $b_j^{(0)}=d_j^{(0)}$. Now set 
$$c_{j}^{(t)}=\frac{1}{t}(d_j^{(t)}-b_j^{(t)}).$$
Note then that $c_{j}^{(0)}=y^{-1}$. It follows that 
$$b_{-i}^{(t)},\ldots,b_{0}^{(t)},c_{-i+\alpha}^{(t)},\ldots,c_0^{(t)},d_1^{(t)},\ldots,d_{r+\alpha}^{(t)}$$
is a basis  of $H^0(\F_r,\CO(iP+Q))$ for $t\neq 0$ and of $H^0(\F_{r+2\alpha},\CO((i-\alpha)P+Q))$ for $t=0$. Furthermore, this basis is compatible with $\pi$ insofar as,
considering $t$ as the deformation parameter, these rational functions in $x, y$ and $t$ give sections of the relevant divisor lifted to $\cX$.
On the other hand, the monomials $x^{-1}$ and $x^{0}=1$  form a basis of both $H^0(\F_r,\CO(P))$ and $H^0(\F_{r+2\alpha},\CO(P))$, which is compatible with $\pi$ in the same sense. For a further discussion of deformations of sections in this context, see \cite[Section 6]{ilten:11e}.

\begin{figure}[htbp]
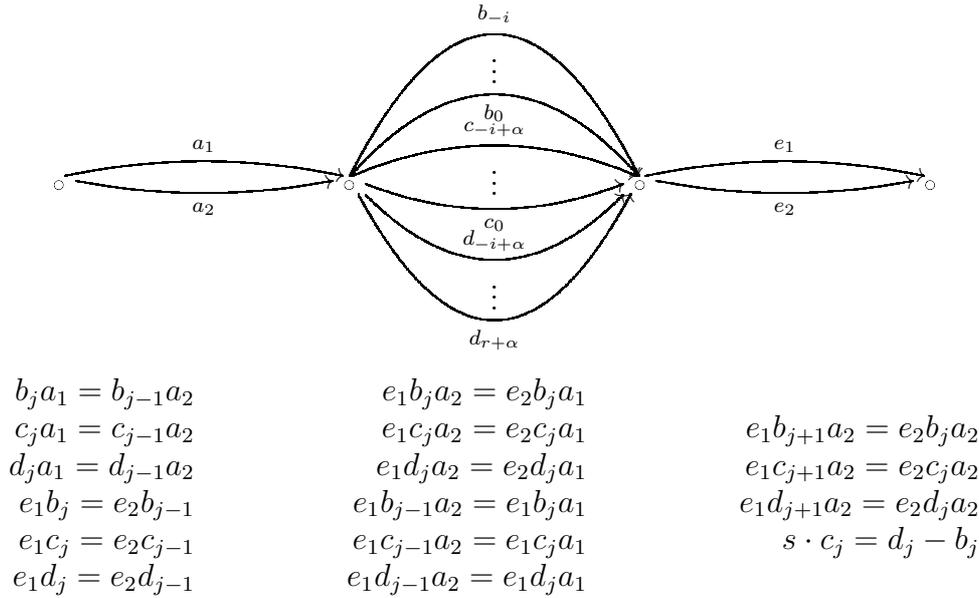

	\begin{center}
	{\quiverex}
	\end{center}
	\caption{A family of quivers}\label{fig:quiverex}
\end{figure}

Figure \ref{fig:quiverex} illustrates a family of quivers. We claim that the corresponding family of path algebras is in fact the desired noncommutative deformation. Indeed, fix some $t\in \base$. The paths $b_j$, $c_j$, and $d_j$ correspond to the global sections $b_j^{(t)}$, $c_j^{(t)}$, $e_j^{(t)}$ of $\CO(iP+Q)$ or $\CO((i-\alpha)P+Q)$, whereas $a_1,a_2,d_1,d_2$ correspond to the global sections $x^{-1},x,x^{-1},x$ of $\CO(P)$. One easily checks that for fixed $t$, the relations in Figure \ref{fig:quiverex}  are precisely the relations needed to represent $\Gamma_t$. Note that in particular, for $t\neq 0$ we can write $c_j$ in terms of $b_j$ and $d_j$, and for $t=0$ we have the relation $b_j=d_j$.

In Section 7 of \cite{perling09}, there is a similar parameterization of path algebras corresponding to Hirzebruch surfaces, which has the advantage that it can contain arbitrarily many Hirzebruch surfaces. However, our construction has the nice attribute that it directly corresponds to a real deformation. The construction we presented here can in fact be easily generalized to construct a noncommutative deformation coming from an arbitrary homogeneous deformation of rational $\CC^*$-surfaces and a strongly exception sequence $\Ex$ on the general fiber degenerating to a strongly exceptional sequence. 

\svjour{
\begin{acknowledgements}
\danksagung
\end{acknowledgements}
}

\bibliographystyle{alpha}
\bibliography{new-e-seq-def}

\end{document}

%% file: newfigures.tex
\newgray{gray1}{.7}
\newgray{gray2}{.74}
\newgray{gray3}{.78}
\newgray{gray4}{.82}
\newgray{gray5}{.86}
\newgray{gray6}{.9}
\newgray{gray7}{.94}
\newgray{gray8}{.98}

\newcommand{\fatdot}{
\psset{dotstyle=o}
\begin{pspicture}(-.1,0)(.1,0)
	\psdot(0,0)
\end{pspicture}
}

\newcommand{\quiverex}{
$$
\begin{array}{c}
	\SelectTips{xy}{10}\xymatrixcolsep{8pc}\xymatrix{
\\
	\\
	\fatdot \ar@/^/[r]^{a_1}\ar@/_/[r]_{a_2}&
	\fatdot\ar@/^4.5pc/[r]^{b_{-i}}_{\vdots}\ar@/^2.6pc/[r]_{b_0}\ar@/^1pc/[r]^{c_{-i+\alpha}}_{\vdots}\ar@/_1pc/[r]_{c_0}\ar@/_2.6pc/[r]^{d_{-i+\alpha}}_{\vdots}\ar@/_4.5pc/[r]_{d_{r+\alpha}} &
\fatdot \ar@/^/[r]^{e_1}\ar@/_/[r]_{e_2}&
\fatdot\\
\\
\\
}\end{array}$$
$$
\begin{array}{r}
			b_ja_1=b_{j-1}a_2\\
			c_ja_1=c_{j-1}a_2\\
			d_ja_1=d_{j-1}a_2\\
			e_1b_j=e_2b_{j-1}\\
			e_1c_j=e_2c_{j-1}\\
			e_1d_j=e_2d_{j-1}\\
		\end{array}
\qquad\qquad
\begin{array}{r}
			e_1b_ja_2=e_2b_ja_1\\
			e_1c_ja_2=e_2c_ja_1\\
			e_1d_ja_2=e_2d_ja_1\\
			e_1b_{j-1}a_2=e_1b_ja_1\\
			e_1c_{j-1}a_2=e_1c_ja_1\\
			e_1d_{j-1}a_2=e_1d_ja_1\\
		\end{array}
\qquad	\qquad
		\begin{array}{r}
		e_1b_{j+1}a_2=e_2b_ja_2\\
			e_1c_{j+1}a_2=e_2c_ja_2\\
			e_1d_{j+1}a_2=e_2d_ja_2\\
			s\cdot c_j=d_j-b_j
		\end{array}
$$
}

\newcommand{\quiverexii}{
$$
\begin{array}{c}
	\SelectTips{xy}{10}\xymatrixcolsep{8pc}\xymatrix{
	\fatdot \ar@/^/[r]^{a_1}\ar@/_/[r]_{a_2}&
	\fatdot&
\fatdot \ar@/^/[r]^{e_1}\ar@/_/[r]_{e_2}&
\fatdot\\
}\end{array}$$
}

\newcommand{\tsgsi}{%
 \psset{unit=.7cm}
 \begin{pspicture}(-7,-1.5)(6,2.5)%
 \psline(-4,0)(-2.5,0)
 \psline(-.5,0)(0,0)
 \psline(0,0)(4,1)
 \psline(-2,1)(-4,0)
 \psline(-2,1)(-.5,1)
 \psline(3.5,1)(4,1)
 \psline{->}(-6.5,0)(-6.5,1)
 \psline{->}(-4,-1)(-3,-1)
 \psset{linestyle=dotted}
 \psline(-2.5,0)(-.5,0)
 \psline(-.5,1)(3.5,1)
 \rput(-7.4,.5){$\big(\frac{y}{y-t}\big)$}
 \rput(-3.5,-1.3){$x$}
 \psdots(-4,0)(-3,0)(0,0)(-2,1)(-1,1)(4,1)
 \rput(-5,0){\scriptsize$(-i,0)$}
 \rput(-3.2,1.1){\scriptsize$(-i+\alpha,1)$}
 \rput(1.1,-.1){\scriptsize$(0,0)$}
 \rput(5,1){\scriptsize$(r+\alpha,1)$}

\end{pspicture}}

\newcommand{\tsgsiii}{%
 \psset{unit=.7cm}
 \begin{pspicture}(-3.5,-1.5)(1,2.5)%
 \psline(-1,0)(0,0)
 \psdots(-1,0)(0,0)
 \psline{->}(1.5,0)(1.5,1)
 \psline{->}(-1,-1)(0,-1)
 \rput(2.4,.5){$\big(\frac{y}{y-t}\big)$}
 \rput(-.5,-1.3){$x$}
 \rput(.6,-.2){\scriptsize$(0,0)$}
 \rput(-1.7,-.2){\scriptsize$(-1,0)$}
\end{pspicture}}

\newcommand{\tsgsiiii}{%
 \psset{unit=.7cm}
 \begin{pspicture}(-3.5,-1.5)(1,2.5)%
 \psline(-1,1)(0,1)
 \psdots(-1,1)(0,1)
 \psline{->}(1.5,0)(1.5,1)
 \psline{->}(-1,-1)(0,-1)
 \rput(2,.5){$y$}
 \rput(-.5,-1.3){$x$}
 \rput(.6,.8){\scriptsize$(0,0)$}
 \rput(-1.7,.8){\scriptsize$(-1,0)$}
\end{pspicture}}

\newcommand{\tsgsii}{%
 \psset{unit=.7cm}
 \begin{pspicture}(-7,-1.5)(6,2.5)%
 \psline(-2,0)(-1.5,0)
 \psline(-.5,0)(0,0)
 \psline(0,0)(4,1)
 \psline(-4,1)(-2,0)
 \psline(-4,1)(-2.5,1)
 \psline(3.5,1)(4,1)
 \psline{->}(-6.5,0)(-6.5,1)
 \psline{->}(-4,-1)(-3,-1)
 \psset{linestyle=dotted}
 \psline(-1.5,0)(-.5,0)
 \psline(-2.5,1)(3.5,1)
 \rput(-7.2,.5){$y$}
 \rput(-3.5,-1.3){$x$}
 \psdots(-2,0)(0,0)(-4,1)(-3,1)(4,1)
 \rput(-5,1){\scriptsize$(-i,0)$}
 \rput(-3.2,-.1){\scriptsize$(-i+\alpha,-1)$}
 \rput(1.1,-.1){\scriptsize$(0,-1)$}
 \rput(5,1){\scriptsize$(r+\alpha,0)$}

\end{pspicture}}

\newcommand{\prooffan}{%
 \psset{unit=.8cm}
 \begin{pspicture}(-5,-3)(4,3)
\psset{linecolor=gray,linestyle=dashed,linewidth=1.5pt}
 \psline{-}(2,1)(-2,1)
\psline{-}(2,-1)(-2,-1)

 \psset{linecolor=black,linestyle=solid,linewidth=1pt}%
 \psline{->}(0,1)
 \psline{->}(0,-1)
 \psline{->}(1,1)
 \psline{->}(1,2)
 \psline{->}(2,3)
 \psline{->}(-1,2)
 \psline{->}(-1,3)

 \rput(0,1.3){$\rho_\alpha$}
 \rput(0,-1.3){$\rho_0$}
 \rput(1.3,1.2){$\rho_1$}
 \rput(-1.3,2){$\rho_l$}
 \rput(-3.1,1){$(\QQ,1)$}
 \rput(-3.1,-1){$(\QQ,-1)$}

\end{pspicture}}

\newcommand{\prooffanb}{%
 \psset{unit=.8cm}
 \begin{pspicture}(-5,-3)(4,3)
\psset{linecolor=gray,linestyle=dashed,linewidth=1.5pt}
 \psline{-}(2,1)(-2,1)
\psline{-}(2,-1)(-2,-1)

 \psset{linecolor=black,linestyle=solid,linewidth=1pt}%
 \psline{->}(0,1)
 \psline{->}(0,-1)
 \psline{->}(1,-1)
 \psline{->}(1,-2)
 \psline{->}(2,-3)
 \psline{->}(-1,2)
 \psline{->}(-1,3)

\rput(0,1.4){$\rho_\alpha'$}
 \rput(0,-1.4){$\rho_0'$}
 \rput(1,-2.2){$\rho_1'$}
 \rput(-1.3,2){$\rho_{l}'$}
 \rput(-3.1,1){$(\QQ,1)$}
 \rput(-3.1,-1){$(\QQ,-1)$}

\end{pspicture}}

\newcommand{\buprooffan}{%
 \psset{unit=.8cm}
 \begin{pspicture}(-5,-3)(4,3)
\psset{linecolor=gray,linestyle=dashed,linewidth=1.5pt}
 \psline{-}(2,1)(-2,1)
\psline{-}(2,-1)(-2,-1)

 \psset{linecolor=black,linestyle=solid,linewidth=1pt}%
 \psline{->}(0,1)
 \psline{->}(1,1)
 \psline{->}(2,1)
 \psline{->}(1,0)
 \psline{->}(1,-1)
 \psline{->}(0,-1)
 \psline{->}(-1,-2)
 \psline{->}(-1,-1)

 \rput(1.2,-1.3){$\rho_0$}
 \rput(1.4,0){$\rho_1$}
 \rput(2.2,1.3){$\rho_2$}
 \rput(-3.1,1){$(\QQ,1)$}
 \rput(-3.1,-1){$(\QQ,-1)$}

\end{pspicture}}
\newcommand{\proofdga}{%
 \psset{unit=1cm}
 \begin{pspicture}(-3.5,-.5)(4,2.5)%
 \psset{linewidth=1pt}%
 \psline{<->}(-2,2)(3,2)
 \psline{<->}(-2,1)(3,1)
 \psline{<->}(-2,0)(3,0)
 \rput(3.5,1){}
 \rput(-2.5,1){$\contract$}
 \rput(3.5,1){$\nocontract$}
 \psset{dotstyle=|}
 \psdots(0,2)(.5,2)(1,2)(2,2)(-1,0)(-.5,0)(0,0)(1,0)
\psdots(0,1)(1,1)
\rput(-3.5,2){$\overline\M_{P_1}$}
\rput(-3.5,0){$\overline\M_{P_2}$}
\rput(-3.5,1){$\overline\M_{s}$}
\rput(0,2.3){\scriptsize$0$}
\rput(.5,2.3){\scriptsize$\frac{1}{2}$}
\rput(1,2.3){\scriptsize$1$}
\rput(2,2.3){\scriptsize$2$}
\rput(-1,-.3){\scriptsize$-1$}
\rput(-.5,-.3){\scriptsize$-\frac{1}{2}$}
\rput(0,-.3){\scriptsize$0$}
\rput(1,-.3){\scriptsize$1$}
\psset{linecolor=black,linestyle=solid,linewidth=.5pt}%
\psline(0,1)(0,2)
\psline(0,1)(.5,2)
\psline(0,1)(1,2)
\psline(0,1)(2,2)
\psline(1,1)(2,2)
\end{pspicture}}

\newcommand{\proofdgb}{%
 \psset{unit=1cm}
 \begin{pspicture}(-3.5,-.5)(4,2.5)%
 \psset{linewidth=1pt}%
 \psline{<->}(-2,2)(3,2)
 \psline{<->}(-2,1)(3,1)
 \psline{<->}(-2,0)(3,0)
 \rput(3.5,1){}
 \rput(3.5,1){$\nocontract$}
 \rput(-2.5,1){$\contract$}
 \psset{dotstyle=|}
 \psdots(0,2)(.5,2)(1,2)(2,2)(-1,0)(-.5,0)(0,0)(1,0)
\psdots(0,1)(1,1)
\rput(-3.5,2){$\overline\M_{P_1}$}
\rput(-3.5,0){$\overline\M_{P_2}$}
\rput(-3.5,1){$\overline\M_{s}$}
\rput(0,2.3){\scriptsize$0$}
\rput(.5,2.3){\scriptsize$\frac{1}{2}$}
\rput(1,2.3){\scriptsize$1$}
\rput(2,2.3){\scriptsize$2$}
\rput(-1,-.3){\scriptsize$-1$}
\rput(-.5,-.3){\scriptsize$-\frac{1}{2}$}
\rput(0,-.3){\scriptsize$0$}
\rput(1,-.3){\scriptsize$1$}
\psset{linecolor=black,linestyle=solid,linewidth=.5pt}%
\psline(0,1)(0,0)
\psline(0,1)(-1,0)
\psline(0,1)(-.5,0)
\psline(0,1)(0,0)
\psline(0,1)(1,0)
\psline(1,1)(1,0)
\end{pspicture}}

\newcommand{\exmulti}{%
 \psset{unit=1cm}
 \begin{pspicture}(-3,-.5)(4,1.5)%
 \psset{linewidth=1pt}%
 \psline{<->}(-2,1)(3,1)
 \psline{<->}(-2,0)(3,0)
 \rput(-2,.5){$\contract$}
 \rput(3,.5){$\nocontract$}
 \rput(3,.5){}
 \psset{dotstyle=|}
 \psdots(0,1)(.5,1)(1,1)(2,1)(-1,0)(-.5,0)(0,0)(1,0)
 \rput(-3,1){$\M_0$}
 \rput(-3,0){$\M_\infty$}
\rput(0,1.3){\scriptsize$0$}	
\rput(.5,1.3){\scriptsize$\frac{1}{2}$}	
\rput(1,1.3){\scriptsize$1$}	
\rput(2,1.3){\scriptsize$2$}	
\rput(-1,-.3){\scriptsize$-1$}	
\rput(-.5,-.3){\scriptsize$-\frac{1}{2}$}	
\rput(0,-.3){\scriptsize$0$}	
\rput(1,-.3){\scriptsize$1$}	
\end{pspicture}}

\newcommand{\exdegdi}{%
 \psset{unit=1cm}
 \begin{pspicture}(-3.5,-.5)(4,2.5)%
 \psset{linewidth=1pt}%
 \psline{<->}(-2,2)(3,2)
 \psline{<->}(-2,0)(3,0)
 \rput(3.5,1){}
 \rput(-2.5,1){$\contract$}
 \rput(3.5,1){$\nocontract$}
 \psset{dotstyle=|}
 \psdots(0,2)(.5,2)(1,2)(2,2)(-1,0)(-.5,0)(0,0)(1,0)
 \rput(-3.5,2){$\M_0$}
 \rput(-3.5,0){$\M_s$}
\rput(0,2.3){\scriptsize$0$}	
\rput(.5,2.3){\scriptsize$\frac{1}{2}$}	
\rput(1,2.3){\scriptsize$1$}	
\rput(2,2.3){\scriptsize$2$}	
\rput(-1,-.3){\scriptsize$-1$}	
\rput(-.5,-.3){\scriptsize$-\frac{1}{2}$}	
\rput(0,-.3){\scriptsize$0$}	
\rput(1,-.3){\scriptsize$1$}	
 \psset{linewidth=.5pt}%
\psline(-1,0)(0,2)(-.5,0)
\psline(0,2)(0,0)(.5,2)
\psline(1,2)(0,0)(2,2)(1,0)
 
\psset{linewidth=1pt}%
\psset{linestyle=dashed,linecolor=gray}%
 \psline{<->}(-2,1)(3,1)
 \rput(-3.5,1){\textcolor{gray}{$\M_0^{(0)}$}}
\end{pspicture}}

\newcommand{\baddefto}{\dashrightarrow}
\newcommand{\degto}{\reflectbox{$\rightsquigarrow$}}
\newcommand{\baddegto}{\dashleftarrow}

\newcommand{\defgraphfour}{
\begin{pspicture}(7,4)
\rput*(4.5,3.5){$dP_6 = X^A_0$}
\rput*{90}(4.5,3){$\baddefto$}
\rput*[l](3.5,2.5){$X^B_0 = X^A_1 \baddegto X^A_2 \degto X^A_r$}
\rput*[r](3,1.5){$X^B_{2r} \defto X^B_2$}
\rput*[l](5,1.5){$X^C_0 = X^B_1 \baddegto X^B_3 \degto X^B_{2r-1}$}
\rput*[r](4.5,0.5){$X^C_r \defto X^C_1$}
\rput*{30}(5,1){$\baddefto$}
\rput*{30}(3.5,2){$\baddefto$}
\rput*{150}(5,2){$\baddefto$}

\rput*[c]{90}(0,2){\small \# exceptional divisors}
\rput*(0.5,0.5){$1$}
\rput*(0.5,1.5){$2$}
\rput*(0.5,2.5){$4$}
\rput*(0.5,3.5){$6$}
\rput*[r](7.5,3.5){\scriptsize$(r>1)$}
\end{pspicture}
}

\newcommand{\defgraphthree}{
\begin{pspicture}(7,4)
\rput*(4,2.75){$X_0 = \dP_7 =X_1$}
\rput*{90}(3,2){$\baddefto$}
\rput*{90}(5,2){$\baddefto$}
\rput*[l](5,1.25){$X_3 \degto X_{2r+1}$}
\rput*[r](3,1.25){$X_{2r} \defto X_2$}

\rput*(0.5,1.25){$2$}
\rput*(0.5,2.75){$3$}
\rput*[c]{90}(0,2){\small \# exceptional divisors}
\rput*[r](7.5,3.5){\scriptsize$(r>1)$}
\end{pspicture}
}

\newcommand{\exdegdibd}{%
 \psset{unit=1cm}
 \begin{pspicture}(-3.5,-.5)(4,2.5)%
 \psset{linewidth=1pt}%
 \psline{<->}(-2,2)(3,2)
 \psline{<->}(-2,0)(3,0)
 \rput(3.5,1){}
 \rput(-2.5,1){$\contract$}
 \rput(3.5,1){$\nocontract$}

 \psset{dotstyle=|}
 \psdots(0,2)(1,2)(2,2)(-1,0)(-.5,0)(0,0)(1,0)
 \rput(-3.5,2){$\M_0$}
 \rput(-3.5,0){$\M_s$}
\rput(0,2.3){\scriptsize$0$}	
\rput(1,2.3){\scriptsize$1$}	
\rput(2,2.3){\scriptsize$2$}	
\rput(-1,-.3){\scriptsize$-1$}	
\rput(-.5,-.3){\scriptsize$-\frac{1}{2}$}	
\rput(0,-.3){\scriptsize$0$}	
\rput(1,-.3){\scriptsize$1$}	
 \psset{linewidth=.5pt}%
\psline(-1,0)(0,2)(-.5,0)
\psline(0,2)(0,0)
\psline(1,2)(0,0)(2,2)(1,0)
\psline[linestyle=dotted](.5,2) 
\psset{linewidth=1pt}%
\psset{linestyle=dashed,linecolor=gray}%
 \psline{<->}(-2,1)(3,1)
 \rput(-3.5,1){\textcolor{gray}{$\M_0^{(0)}$}}
\end{pspicture}}

\newcommand{\exdegdibua}{%
 \psset{unit=1cm}
 \begin{pspicture}(-3.5,-.5)(4,2.5)%
 \psset{linewidth=1pt}%
 \psline{<->}(-2,2)(3,2)
 \psline{<->}(-2,0)(3,0)
 \rput(3.5,1){}
 \rput(-2.5,1){$\contract$}
 \rput(3.5,1){$\nocontract$}
 \psset{dotstyle=|}
 \psdots(0,2)(.5,2)(1,2)(2,2)(-1,0)(-.5,0)(0,0)(1,0)(2,0)
 \rput(-3.5,2){$\M_0$}
 \rput(-3.5,0){$\M_s$}
\rput(0,2.3){\scriptsize$0$}	
\rput(.5,2.3){\scriptsize$\frac{1}{2}$}	
\rput(1,2.3){\scriptsize$1$}	
\rput(2,2.3){\scriptsize$2$}	
\rput(-1,-.3){\scriptsize$-1$}	
\rput(-.5,-.3){\scriptsize$-\frac{1}{2}$}	
\rput(0,-.3){\scriptsize$0$}	
\rput(1,-.3){\scriptsize$1$}
\rput(2,-.3){\scriptsize$2$}	

 \psset{linewidth=.5pt}%
\psline(-1,0)(0,2)(-.5,0)
\psline(0,2)(0,0)(.5,2)
\psline(1,2)(0,0)(2,2)(1,0)
\psline[linestyle=dotted](2,2)(2,0) 
\psset{linewidth=1pt}%
\psset{linestyle=dashed,linecolor=gray}%
 \psline{<->}(-2,1)(3,1)
 \rput(-3.5,1){\textcolor{gray}{$\M_0^{(0)}$}}
\end{pspicture}}

\newcommand{\exdegdibub}{%
 \psset{unit=1cm}
 \begin{pspicture}(-3.5,-.5)(4,2.5)%
 \psset{linewidth=1pt}%
 \psline{<->}(-2,2)(3,2)
 \psline{<->}(-2,0)(3,0)
 \rput(3.5,1){}
 \rput(-2.5,1){$\contract$}
 \rput(3.5,1){$\nocontract$}
 \psset{dotstyle=|}
 \psdots(0,2)(.5,2)(1,2)(2,2)(3,2)(-1,0)(-.5,0)(0,0)(1,0)
 \rput(-3.5,2){$\M_0$}
 \rput(-3.5,0){$\M_s$}
\rput(0,2.3){\scriptsize$0$}	
\rput(.5,2.3){\scriptsize$\frac{1}{2}$}	
\rput(1,2.3){\scriptsize$1$}	
\rput(2,2.3){\scriptsize$2$}	
\rput(3,2.3){\scriptsize$3$}	
\rput(-1,-.3){\scriptsize$-1$}	
\rput(-.5,-.3){\scriptsize$-\frac{1}{2}$}	
\rput(0,-.3){\scriptsize$0$}	
\rput(1,-.3){\scriptsize$1$}	
 \psset{linewidth=.5pt}%
\psline(-1,0)(0,2)(-.5,0)
\psline(0,2)(0,0)(.5,2)
\psline(1,2)(0,0)(2,2)(1,0)
\psline[linestyle=dotted](1,0)(3,2) 
\psset{linewidth=1pt}%
\psset{linestyle=dashed,linecolor=gray}%
 \psline{<->}(-2,1)(3,1)
 \rput(-3.5,1){\textcolor{gray}{$\M_0^{(0)}$}}
\end{pspicture}}

\newcommand{\figDPthreeA}{
 \psset{unit=.65cm}
\begin{pspicture}(12,9)
\psset{dotstyle=|}
\rput*[b](3,7){\scalebox{0.6}{$D_1$}}
\rput*[l](5,5){\scalebox{0.6}{$D_2$}}
\rput*[tl](4.5,3.5){\scalebox{0.6}{$D_3$}}
\rput*[t](3,3){\scalebox{0.6}{$D_4$}}
\rput*[r](1,5){\scalebox{0.6}{$D_5$}}
\psline[linecolor=lightgray](1,5)(5,5)
\psline[linecolor=lightgray](3,3)(3,7)
\psline[linecolor=lightgray](3,5)(4.5,3.5)
\psline(2.25,6.75)(4.75,4.25)
\psline(1.25,5.75)(3.75,3.25)
\rput*[t](3,0.8){\scalebox{0.5}{$0$}}
\rput*[b](4,2.2){\scalebox{0.5}{$1$}}
\rput*[t](2,0.8){\scalebox{0.5}{$-1$}}
\rput*[b](3,2.2){\scalebox{0.5}{$0$}}
\psline(1,1)(5,1)
\psline(1,2)(5,2)
\psline(2,1)(3,2)(3,1)(4,2)
\psdots(2,1)(3,1)(3,2)(4,2)
\rput*(6,8){\scalebox{0.85}{$\dP_7 = X_0 \baddegto X_2$}}
\rput*[b](9,7){\scalebox{0.6}{$D_1$}}
\rput*[l](10.5,8){\scalebox{0.6}{$D_2$}}
\rput*[l](10.75,6.75){\scalebox{0.6}{$D_3$}}
\rput*[t](9,3){\scalebox{0.6}{$D_4$}}
\rput*[r](7,5){\scalebox{0.6}{$D_5$}}
\psline[linecolor=lightgray](7,5)(9,5)
\psline[linecolor=lightgray](9,3)(9,7)
\psline[linecolor=lightgray](9,5)(10.75,6.75)
\psline[linecolor=lightgray](9,5)(10.5,8)
\psline(8.25,3.25)(10.75,5.75)
\psline(7.25,4.25)(10.5,7.5)
\psline(7,1)(11,1)
\psline(7,2)(11,2)
\psdots(8,2)(9,2)(9,1)(10,2)
\rput*(0.75,1.5){\scalebox{0.75}{$\nocontract$}}
\rput*(5.25,1.5){\scalebox{0.75}{$\contract$}}
\rput*(6.75,1.5){\scalebox{0.75}{$\nocontract$}}
\rput*(11.25,1.5){\scalebox{0.75}{$\contract$}}
\end{pspicture}
}

\newcommand{\figDPthreeB}{
 \psset{unit=.65cm}
\begin{pspicture}(12,8)
\psset{dotstyle=|}
\rput*[b](3,7){\scalebox{0.6}{$D_1$}}
\rput*[bl](4.5,6.5){\scalebox{0.6}{$D_2$}}
\rput*[l](5,5){\scalebox{0.6}{$D_3$}}
\rput*[t](3,3){\scalebox{0.6}{$D_4$}}
\rput*[r](1,5){\scalebox{0.6}{$D_5$}}
\psline[linecolor=lightgray](1,5)(5,5)
\psline[linecolor=lightgray](3,3)(3,7)
\psline[linecolor=lightgray](3,5)(4.5,6.5)
\psline(2.25,6.75)(4.75,4.25)
\psline(1.25,5.75)(3.75,3.25)
\rput*[b](3,2.2){\scalebox{0.5}{$0$}}
\rput*[b](3.5,2.2){\scalebox{0.5}{$\frac{1}{2}$}}
\rput*[b](4,2.2){\scalebox{0.5}{$1$}}
\rput*[tr](2,0.8){\scalebox{0.5}{$-1$}}
\rput*[t](3,0.8){\scalebox{0.5}{$0$}}
\psline(1,1)(5,1)
\psline(1,2)(5,2)
\psline(2,1)(3,2)(3,1)(4,2)
\psline(3,1)(3.5,2)
\psdots(2,1)(3,1)(3,2)(3.5,2)(4,2)
\rput*(6,8){\scalebox{0.85}{$\dP_7 = X_1 \baddegto X_3$}}
\rput*[b](9,7){\scalebox{0.6}{$D_1$}}
\rput*[b](10,8){\scalebox{0.6}{$D_2$}}
\rput*[b](10.25,7.5){\scalebox{0.6}{$D_3$}}
\rput*[t](9,3){\scalebox{0.6}{$D_4$}}
\rput*[r](7,5){\scalebox{0.6}{$D_5$}}
\psline[linecolor=lightgray](7,5)(9,5)(10.25,7.5)
\psline[linecolor=lightgray](9,3)(9,7)
\psline[linecolor=lightgray](9,5)(10,8)
\psline(7.25,4.25)(10.5,7.5)
\psline(8.25,3.25)(10.75,5.75)
\psline(7,1)(11,1)
\psline(7,2)(11,2)
\psdots(8,2)(9,1)(9,2)(9.5,2)(10,2)
\rput*(0.75,1.5){\scalebox{0.75}{$\nocontract$}}
\rput*(5.25,1.5){\scalebox{0.75}{$\nocontract$}}
\rput*(6.75,1.5){\scalebox{0.75}{$\nocontract$}}
\rput*(11.25,1.5){\scalebox{0.75}{$\nocontract$}}
\end{pspicture}
}

\newcommand{\TVPtwoA}{
\psset{unit=1ex}
\psset{linewidth=0.4pt}
\begin{pspicture}[shift=-0.25ex](2,2)
\psline(1,1)(1.9,1.9)
\psline(0,1)(1,1)
\psline(1,0)(1,1)
\end{pspicture}
}

\newcommand{\TVPtwoB}{
\psset{unit=1ex}
\psset{linewidth=0.4pt}
\begin{pspicture}[shift=-0.25ex](2,2)
\psline(1,1)(1.9,0.1)
\psline(0,1)(1,1)
\psline(1,2)(1,1)
\end{pspicture}
}

\newcommand{\TVFoneA}{
\psset{unit=1ex}
\psset{linewidth=0.4pt}
\begin{pspicture}[shift=-0.25ex](2,2)
\psline(1,1)(1.9,1.9)
\psline(0,1)(2,1)
\psline(1,0)(1,1)
\end{pspicture}
}

\newcommand{\TVFoneB}{
\psset{unit=1ex}
\psset{linewidth=0.4pt}
\begin{pspicture}[shift=-0.25ex](2,2)
\psline(1,1)(1.9,0.1)
\psline(0,1)(2,1)
\psline(1,2)(1,1)
\end{pspicture}
}

\newcommand{\TVdPtA}{
\psset{unit=1ex}
\psset{linewidth=0.4pt}
\begin{pspicture}[shift=-0.25ex](2,2)
\psline(1,1)(1.9,1.9)
\psline(0,1)(2,1)
\psline(1,0)(1,2)
\end{pspicture}
}

\newcommand{\TVdPtB}{
\psset{unit=1ex}
\psset{linewidth=0.4pt}
\begin{pspicture}[shift=-0.25ex](2,2)
\psline(1,1)(1.9,0.1)
\psline(0,1)(2,1)
\psline(1,0)(1,2)
\end{pspicture}
}